\documentclass[10pt,a4paper]{amsart}
\usepackage{amsmath,amssymb,color,multicol,setspace}
\usepackage[utf8,utf8x]{inputenc}

\usepackage[pdftitle = {Frieze\ patterns\ over\ integers\ and\ other\ subsets\ of\ the\ complex\ numbers}]{hyperref}
\usepackage{longtable}
\usepackage{xy,amscd} 
\usepackage{epsfig}
\usepackage{rotating}
\usepackage{xspace}
\usepackage{float}
\usepackage{enumitem} 
\xyoption{all}
\usepackage{tikz}
\usetikzlibrary{arrows,decorations.pathmorphing,decorations.pathreplacing,positioning,shapes.geometric,shapes.misc,decorations.markings,decorations.fractals,calc,patterns}

\newtheorem{propo}{Proposition}[section]
\newtheorem{corol}[propo]{Corollary}
\newtheorem{theor}[propo]{Theorem}
\newtheorem{lemma}[propo]{Lemma}

\theoremstyle{definition}
\newtheorem{defin}[propo]{Definition}
\newtheorem{examp}[propo]{Example}

\theoremstyle{remark}
\newtheorem{remar}[propo]{Remark}

\numberwithin{equation}{section}

\newcommand{\NN }{\mathbb{N}}
\newcommand{\CC }{\mathbb{C}}

\newcommand{\ZZ }{\mathbb{Z}}

\newcommand{\ii }{\mathrm{i}}

\DeclareMathOperator{\SL}{SL}

\DeclareMathOperator{\md}{mod}

\newcommand{\Fc }{\mathcal F}

\title[Frieze patterns over integers and other subsets of the complex numbers]
{Frieze patterns over integers and other\\ subsets of the complex numbers}

\author{Michael~Cuntz}
\address{Michael Cuntz, Leibniz Universit\"at Hannover,
Institut f\"ur Algebra, Zahlentheorie und Diskrete Mathematik,
Fakult\"at f\"ur Mathematik und Physik,
Welfengarten 1,
D-30167 Hannover, Germany}
\email{cuntz@math.uni-hannover.de}
\urladdr{https://www.iazd.uni-hannover.de/cuntz.html}

\author{Thorsten~Holm}
\address{Thorsten Holm, Leibniz Universit\"at Hannover,
Institut f\"ur Algebra, Zahlentheorie und Diskrete Mathematik,
Fakult\"at f\"ur Mathematik und Physik,
Welfengarten 1,
D-30167 Hannover, Germany}
\email{holm@math.uni-hannover.de}
\urladdr{http://www2.iazd.uni-hannover.de/\~{ }tholm}



\begin{document}

\begin{abstract}
We study (tame) frieze patterns over subsets of the complex numbers, with particular emphasis on the corresponding quiddity cycles. We provide new general transformations for quiddity cycles of frieze patterns. As one application, we present a combinatorial model for obtaining the quiddity cycles of all tame frieze patterns over the integers (with zero entries allowed),
generalising the classic Conway-Coxeter theory. This model is thus also a model for the set of specializations of cluster algebras of Dynkin type $A$ in which all cluster variables are integers.

Moreover, we address the question of whether for a given height there are only finitely many non-zero frieze patterns over a given subset $R$ of the complex numbers. Under certain conditions on $R$, we show upper bounds for the absolute values of entries in the quiddity cycles. As a consequence, we obtain that if $R$ is a discrete subset of the complex numbers then for every height there are only finitely many non-zero frieze patterns over $R$. Using this, we disprove a conjecture of Fontaine,
by showing that for a complex $d$-th root of unity $\zeta_d$ there are only finitely many non-zero frieze patterns for a given height over $R=\mathbb{Z}[\zeta_d]$ if and only if $d\in \{1,2,3,4,6\}$.
\end{abstract}

\maketitle

%%%%%%%%%%%%%%%%%%%%%%%%%%%%%%%%%%%%%%%%%%%%%%%%%%%%%%%%%%%%%%%%%%%%%%%%%%%%%%%%%%%
%%%%%%%%%%%%%%%%%%%%%%%%%%%%%%%%%%%%%%%%%%%%%%%%%%%%%%%%%%%%%%%%%%%%%%%%%%%%%%%%%%%
\section{Introduction}

Frieze patterns have been introduced by Coxeter \cite{Cox71}. Shortly afterwards,
Conway and Coxeter presented a beautiful theory for frieze patterns over natural numbers \cite{CC73},
among other things showing that there is a bijection between frieze patterns with entries in $\mathbb{N}$
and triangulations of polygons. This result has recently been generalized to $p$-angulations of polygons
and certain frieze patterns over positive real numbers \cite{HJ17}.

Since the invention of cluster algebras by Fomin and Zelevinsky around 2000, 
frieze patterns have attracted renewed interest because of a close connection to 
cluster algebras. In fact, if one allows Laurent polynomials as entries of a frieze pattern
then starting with a set of indeterminates produces the cluster variables of the cluster algebra of
Dynkin type $A$ as entries in the frieze pattern.

So it is very natural to consider frieze patterns over other 
sets of numbers than integers. In this paper we start out by considering 
frieze patterns with entries in subsets of the field of
complex numbers.
In general there are far too many frieze patterns, see for instance \cite{Cuntz17},
and a general frieze pattern does not share any periodicity properties. Therefore, one usually 
restricts to studying tame frieze patterns \cite{BR10}. This is a very large class of frieze
patterns including basically
all interesting classes studied so far (like friezes corresponding to cluster algebras, the Conway-Coxeter
friezes etc.), sharing nice symmetry properties and allowing unified methods to be applied to study them.
For more details on frieze patterns we refer the reader to a nice survey by Morier-Genoud \cite{MG15}.

However, it is still very subtle to describe all tame frieze patterns 
with entries from a given set $R$ of numbers. For $R=\mathbb{N}$ this is the classic Conway-Coxeter
theory. Frieze patterns over $R=\mathbb{Z}\setminus \{0\}$ have been described by Fontaine
\cite{F14}; here it turns out that only very few new frieze patterns appear in addition to the 
ones over $\mathbb{N}$. The situation changes drastically for tame frieze patterns over $\mathbb{Z}$,
i.e.\ when zeroes are allowed as entries. Then infinitely many new frieze patterns appear. As one of the 
main results of this paper we show how every such frieze pattern can be obtained from a new
combinatorial model. This combinatorial model generalizes the combinatorial model via triangulations from the Conway-Coxeter theory (moreover it can be seen as a special case of the model proposed in
a later paper \cite{C17c}). See Section \ref{sec:Zfriezes} 
for the description of frieze patterns over $\mathbb{Z}$ and Section \ref{sec:combmodel} 
for the corresponding combinatorial model.

As main tools for achieving this, we provide new general transformations on frieze patterns, or more precisely
on quiddity cycles, which might turn out to be useful in other situations as well.
See Section \ref{sec:transform} for precise statements. These transformations substantially generalise
the classic Conway-Coxeter theory, where for inductively 
proving the bijection between frieze patterns over $\mathbb{N}$ and 
triangulations of polygons one only needs transformations which insert/remove an entry 1 in the quiddity
cycle. Over other subsets $R\subseteq \mathbb{C}$, quiddity cycles do not necessarily contain a 1.
However, we can show in Section \ref{sec:bounds} that each quiddity cycle over the complex 
numbers contains an entry (even two entries) of absolute value less than 2; see Corollary 
\ref{cor:small}.

The possible absence of 1's in quiddity cycles is one of the reasons why combinatorially describing all frieze patterns 
over a given set $R$ can be quite hard. It is even a non-trivial problem to decide whether for a given height
there are finitely or infinitely many (non-zero) frieze patterns over $R$.  

As one main result of this paper we provide in Theorem \ref{prop:bound}
a criterion for having finitely many non-zero
frieze patterns of any given height
over certain subsets $R$. As we show in Corollary \ref{cor:discrete}, this criterion implies 
that for any discrete subset $R\subseteq \mathbb{C}$ there are only finitely many non-zero frieze patterns 
over $R$ for each height.

As an application we can disprove a conjecture by Fontaine \cite[Conjecture 6.2]{F14}
on frieze patterns over rings of the form 
$R= \mathbb{Z}[\zeta_d]$ where $\zeta_d\in\CC$ is a primitive $d$-th root of unity.
Namely, we show in Corollary \ref{cor:zeta_d} that there are only finitely many 
non-zero frieze patterns of height $n$ over $\mathbb{Z}[\zeta_d]$ if and only if $d\in\{1,2,3,4,6\}$
(independent of the height of the frieze patterns).  

\medskip

\begin{center}
{\sc Acknowledgement}
\end{center}
We thank Sophie Morier-Genoud for answering questions about some aspects of this paper. 
We are grateful to the referee for a very knowledgable report and many useful suggestions.

%%%%%%%%%%%%%%%%%%%%%%%%%%%%%%%%%%%%%%%%%%%%%%%%%%%%%%%%%%%%%%%%%%%%%%%%%%%%%%%%%
%%%%%%%%%%%%%%%%%%%%%%%%%%%%%%%%%%%%%%%%%%%%%%%%%%%%%%%%%%%%%%%%%%%%%%%%%%%%%%%%%
\section{Quiddity cycles}

In this section we collect some fundamental definitions and results which are later needed in the 
paper.

\begin{defin} \label{def:frieze}
Let $R\subseteq\CC$ be a subset of the complex numbers.
\smallskip

\noindent
(1) A \emph{frieze pattern} over $R$ 
is an array $\Fc$ of the form
\[
\begin{array}{ccccccccccc}
 & & \ddots & & & &\ddots  & & & \\
0 & 1 & c_{i-1,i+1} & c_{i-1,i+2} & \cdots & \cdots & c_{i-1,n+i} & 1 & 0 & & \\
& 0 & 1 & c_{i,i+2} & c_{i,i+3} & \cdots & \cdots & c_{i,n+i+1} & 1 & 0 & \\
& & 0 & 1 & c_{i+1,i+3} & c_{i+1,i+4} & \cdots & \cdots & c_{i+1,n+i+2} & 1 & 0 \\
 & & & & \ddots  & & & &\ddots  & 
\end{array}
\]
where $c_{i,j}$ are numbers in $R$, and such that every
(complete) adjacent $2\times 2$ submatrix has determinant $1$.
We call $n$ the \emph{height} of the frieze pattern $\Fc$.
We say that the frieze pattern $\Fc$ is
\emph{periodic} with period $m>0$ if $c_{i,j}=c_{i+m,j+m}$ for all $i,j$.
\smallskip

\noindent
(2) A frieze pattern is called {\em tame} if every adjacent $3\times 3$-submatrix has
determinant 0.
\end{defin}

Frieze patterns with non-zero entries are always tame, due to Sylvester's theorem, see for example \cite{Cuntz17}.

Frieze patterns have been introduced by Coxeter \cite{Cox71} and studied further by Conway and Coxeter \cite{CC73}. More precisely, Conway and Coxeter studied frieze patterns over $\mathbb{N}$, i.e.\ frieze patterns with positive
integral entries. If one allows the entries in a frieze pattern to be rational functions over
$\mathbb{Q}$, still observing the local condition on $2\times 2$-determinants, then starting with a set of 
indeterminates, one obtains the cluster variables of Fomin and Zelevinsky's cluster algebras of Dynkin type $A$ as 
entries in the frieze pattern. This is one instance, among others, which shows that for applications
of frieze patterns to other areas of mathematics (e.g.\ geometry, representation theory, integrable systems)
it is useful to allow entries from various (semi-)rings.

\begin{examp} \label{ex:friezes}
\ \\
(1) {\em Conway-Coxeter frieze patterns} are exactly the frieze patterns over $\mathbb{N}$.
An intriguing feature of these frieze patterns is that there is a bijection between 
the frieze patterns of height $n$ and the triangulations of a regular $(n+3)$-gon.
In particular, every frieze pattern over $\mathbb{N}$ of height $n$ is periodic with 
period $n+3$. The following is an example:
\begin{figure}[h]
  \centering
  \begin{minipage}[b]{0.4\textwidth}
 $$   \begin{array}{cccccccccccc}
 &  & \ddots &  &  &  &  &  &  &  &  & \\
 &  &  & &  &  &  &  &  &  &  & \\
 0 & 1 & 1 & 3 & 2 & 1 & 0 &  &  &  &  & \\
& 0 & 1 & 4 & 3 & 2 & 1 & 0 &  & & & \\
 & & 0 & 1 & 1 & 1 & 1 & 1 & 0 & & & \\
 & &  & 0 & 1 & 2 & 3 & 4 & 1 & 0 & & \\
 & &  &  & 0 & 1 & 2 & 3 & 1 & 1 & 0 & \\
 & &  &  &  & 0 & 1 & 2 & 1 & 2 & 1 & 0 \\
 &  &  &  &  &  &  &  &  & \ddots &  &
\end{array}
$$
  \end{minipage}
  \hfill
  \begin{minipage}[b]{0.4\textwidth}
    \begin{tikzpicture}[auto,baseline=(s.center)]
    \node[name=s, draw, shape=regular polygon, regular polygon sides=6, minimum size=2.8cm] {};
    \draw[thick] (s.corner 1) to (s.corner 3);
    \draw[thick] (s.corner 6) to (s.corner 3);
    \draw[thick] (s.corner 5) to (s.corner 3);
   \draw[shift=(s.corner 1)]  node[above]  {{\small 2}};
  \draw[shift=(s.corner 2)]  node[above]  {{\small 1}};
  \draw[shift=(s.corner 3)]  node[left]  {{\small 4}};
  \draw[shift=(s.corner 4)]  node[below]  {{\small 1}};
  \draw[shift=(s.corner 5)]  node[below]  {{\small 2}};
  \draw[shift=(s.corner 6)]  node[right]  {{\small 2}};
   \end{tikzpicture}
  \end{minipage}
\end{figure}
The numbers at the vertices of the hexagon are the numbers of triangles attached; and these numbers 
(in counterclockwise order) yield the first diagonal in the corresponding frieze pattern on the left.    
\smallskip

\noindent
(2) The array
\[ \begin{array}{rrrrrrrrrrrr}
&  &  & \ddots &  &  &  &  & & & & \\
0 & 1 & -\ii + 1 & 1 & \ii + 1 & 1 & 0 &    &    &    &    &   \\
   & 0 & 1 & \ii + 1 & 2\ii + 1 & 2 & 1 & 0 &    &    &    &   \\
   &    & 0 & 1 & 2 & -2\ii + 1 & -\ii + 1 & 1 & 0 &    &    &   \\
   &    &    & 0 & 1 & -\ii + 1 & 1 & \ii + 1 & 1 & 0 &    &   \\
   &    &    &    & 0 & 1 & \ii + 1 & 2\ii + 1 & 2 & 1 & 0 &   \\
   &    &    &    &    & 0 & 1 & 2 & -2\ii + 1 & -\ii + 1 & 1 & 0\\
&  &  &  &  &  &  &  &\ddots & & &
\end{array} \]
repeated infinitely many times to both sides,
is a frieze pattern over the Gaussian integers $\ZZ[\ii]$; it is periodic with period 6.
\smallskip

\noindent
(3)
The fact that the frieze patterns in the above example are periodic, follows from some general 
results on frieze patterns. In fact, if all entries $c_{i,j}$ in a frieze pattern of height $n$
are non-zero, then the frieze pattern is periodic with period $n+3$; see Proposition \ref{perck}
below for details. A frieze pattern with zero entries might not be periodic at all. For instance,
for every sequences $(a_i)_{i\in\mathbb{Z}}$ and $(b_i)_{i\in\mathbb{Z}}$ we have a frieze
pattern of the form
\[
\begin{array}{cccccccccccc}
 &  &  & \ddots &  &  &  &  &  &  &  & \\
 &  &  & &  &  &  &  &  &  &  & \\
 0 & 1 & a_1 & -1 & b_1 & 1 & 0 &  &  &  &  & \\
& 0 & 1 & 0 & -1 & 0 & 1 & 0 &  & & & \\
 & & 0 & 1 & a_2 & -1 & b_2 & 1 & 0 & & & \\
 & &  & 0 & 1 & 0 & -1 & 0 & 1 & 0 & & \\
 & &  &  & 0 & 1 & a_3 & -1 & b_3 & 1 & 0 & \\
 & &  &  &  & 0 & 1 & 0 & -1 & 0 & 1 & 0 \\
 &  &  &  &  &  &  &  &  & \ddots &  &
\end{array}
\]
\end{examp}

It follows directly from the definition that every non-zero frieze pattern is uniquely determined 
by the entries $c_{i,i+2}$ in the first diagonal of the frieze pattern (cf.\ Definition \ref{def:frieze});
in fact, if all entries are non-zero then one can use the condition for the $2\times 2$-determinants
to be 1 to compute successively the second diagonal, the third diagonal etc.
(But note that the above Example \ref{ex:friezes}\,(3) shows that this is no longer true 
if zero entries appear.) 
\smallskip

It is a priori not clear which sequences of numbers actually yield a frieze pattern. For 
getting a criterion we need the following matrices which play a crucial role in the theory of 
frieze patterns:

\begin{defin} \label{def:etamatrix}
For $c\in \CC$, let
$\eta(c) = \begin{pmatrix} c & -1 \\ 1 & 0 \end{pmatrix}$. Moreover, for any sequence
$c_1,\ldots,c_{\ell}$ of complex numbers and any $1\le i-1\le j\le \ell$ we set 
$M_{i,j} = \prod_{k=i}^{j} \eta(c_k).
$
\end{defin}

The following result collects several fundamental properties of the above matrices and explains
their importance in the context of frieze patterns. These results are known but not well 
documented in the literature, thus we give a precise statement and include a proof.
\smallskip

As usual, for any matrix $A$ we denote by $A_{r,s}$ the entry in row $r$ and column $s$.

\begin{propo}\label{perck} 
Let $R\subseteq \mathbb{C}$ be a subset.  
\begin{enumerate}
\item
Let $\Fc$ be a frieze pattern over $R$ of height $n$, with entries $c_{i,j}$ as in Definition
\ref{def:frieze}. We set  
$c_k:=c_{k,k+2}$ for $k\in\ZZ$.
If 
$\Fc$ is a tame frieze pattern then $\Fc$ is
periodic with period $m=n+3$.
Furthermore,
\[
\prod_{k=1}^{m} \eta(c_k) = 
\begin{pmatrix} -1 & 0 \\ 0 & -1 \end{pmatrix},
\quad\text{and}\quad
c_{i,j+2} = (M_{i,j})_{1,1}. \]
(Notice that we may assume $j\ge i$ in the product defining $M_{i,j}$ by possibly adding multiples of $2m$ to $j$.)
\item Suppose that $(c_1,\ldots,c_m)\in R^m$ satisfies 
$\prod_{k=1}^{m} \eta(c_k) = \begin{pmatrix} -1 & 0 \\ 0 & -1 \end{pmatrix}$. We define 
$c_k$ for all $k\in\mathbb{Z}$ by repeating the sequence $(c_1,\ldots,c_m)$ periodically.  
Then the array
\[ (a_{i,j+2})_{i,j} = ((M_{i,j})_{1,1})_{i,j} \]
(where $i-1\le j\le m+i-3$)
defines a periodic frieze pattern over $\mathbb{C}$ with period $m$ and height $m-3$.
If $R$ is a ring, then the entries of this frieze pattern are in $R$.
Moreover, this frieze pattern is tame.  
\end{enumerate}
\end{propo}

\begin{proof}
(1)
Let $\Fc=(c_{i,j})$ be a frieze pattern over $R$.
Consider an adjacent $3\times 3$-submatrix $M$ of $\Fc$.
The first two columns of $M$ cannot be linearly dependent because the upper left $2\times 2$-submatrix has determinant $1$.
But then since $\Fc$ is tame, the determinant of $M$ is zero, so
\[ M =
\begin{pmatrix}
a & b & sa+tb \\
c & d & sc+td \\
e & f & se+tf
\end{pmatrix}
\]
for suitable $a,b,c,d,e,f,s,t$.
Now the fact that all adjacent $2\times 2$-determinants are $1$ implies
$$1=b(sc+td)-d(sa+tb)=s(bc-ad)=-s,$$
so $s=-1$.
Thus setting $a=c_{i,j}$, $b=c_{i,j+1}$
we see that for fixed $j$, there is a $t_j$ such that
\begin{equation}\label{nextcol}
\eta(t_j) \begin{pmatrix} c_{i,j+1} \\ c_{i,j} \end{pmatrix} =
\begin{pmatrix} -c_{i,j}+t_j c_{i,j+1} \\ c_{i,j+1} \end{pmatrix} =
\begin{pmatrix} c_{i,j+2} \\ c_{i,j+1} \end{pmatrix}
\end{equation}
for all $i$.
If we extend the frieze by a row of $0$'s and $-1$'s on both sides, then the $\SL_2$-condition is still satisfied and in each consecutive pair of rows we find
$\begin{pmatrix} 0 & 1 \\ -1 & 0 \end{pmatrix}$ at the beginning and $\begin{pmatrix} 0 & -1 \\ 1 & 0 \end{pmatrix}$ at the end of the frieze pattern.
Since next to the left matrix there is $\begin{pmatrix} 1 & c_{i,i+2} \\ 0 & 1 \end{pmatrix}$, Equation (\ref{nextcol}) implies $t_i=c_{i,i+2}$.
Moreover, from this it is clear that the product of the matrices $\eta(c_{k,k+2})$, $k=i,\ldots,i+m-1$ (from left to right) in each row $i$ is the negative of the identity matrix.
Further, by Equation (\ref{nextcol}), the entries in the frieze pattern are just the top left entries in the intermediate products.
For the periodicity notice that by the above consideration, comparing two consecutive rows gives sequences $t_1,\ldots,t_m$ and $t_2,\ldots,t_{m+1}$ such that
\[ \prod_{i=1}^m \eta(t_i) = \prod_{i=2}^{m+1} \eta(t_i) = 
\begin{pmatrix} -1 & 0 \\ 0 & -1 \end{pmatrix}, \]
which implies $t_1=t_{m+1}$ and hence $c_{i,i+2}=c_{i+m,i+m+2}$ for all $i$. Since all entries in $\Fc$ are determined by these numbers, the whole frieze is periodic.
\smallskip

\noindent
(2) There are several items to check, namely the determinants of the adjacent $2\times 2$-submatrices, the tameness,
the height (i.e.\ that the pattern ends with a diagonal of 1's), and the periodicity.

We start by considering the determinants of adjacent 
$2\times 2$-submatrices.
Because of $$\eta(c_i) M_{i+1,j} = M_{i,j}$$ we have $(M_{i+1,j})_{1,1}=(M_{i,j})_{2,1}$. Similarly, we obtain
\begin{equation}\label{entries}
(M_{i+1,j})_{1,1}=-(M_{i+1,j+1})_{1,2}=(M_{i,j})_{2,1}=-(M_{i,j+1})_{2,2}.
\end{equation}
From this we get
\begin{eqnarray*}
\det \begin{pmatrix} a_{i,j+1} & a_{i,j+2} \\ a_{i+1,j+1} & a_{i+1,j+2} \end{pmatrix}
&=& \det \begin{pmatrix} (M_{i,j-1})_{1,1} & (M_{i,j})_{1,1} \\ (M_{i+1,j-1})_{1,1} & (M_{i+1,j})_{1,1} \end{pmatrix} \\
&=& \det \begin{pmatrix} -(M_{i,j})_{1,2} & (M_{i,j})_{1,1} \\ -(M_{i,j})_{2,2} & (M_{i,j})_{2,1} \end{pmatrix}
= \det M_{i,j} = 1.
\end{eqnarray*}
Thus, the proposed array satisfies the
condition for the entries of a frieze pattern.
Equation (\ref{entries}) and $M_{i,j+1}=M_{i,j}\eta(c_{j+1})$ also give
\begin{eqnarray*}
\begin{pmatrix}
a_{i,j+1} & a_{i,j+2} & a_{i,j+3}
\end{pmatrix}
&=&
\begin{pmatrix}
-(M_{i,j})_{1,2} & (M_{i,j})_{1,1} & (M_{i,j+1})_{1,1}
\end{pmatrix}
\\ &=&
\begin{pmatrix}
-(M_{i,j})_{1,2} & (M_{i,j})_{1,1} & (M_{i,j})_{1,2} + c_{j+1}(M_{i,j})_{1,1}
\end{pmatrix}.
\end{eqnarray*}
This shows that any three consecutive columns in the array are linear dependent,
hence the frieze pattern is tame.

For showing that the array is a frieze pattern of height $m-3$, we have to show that 
the pattern has a bounding diagonal of 1's. With the labelling
as in Definition \ref{def:frieze} this means that we have to show that
$a_{i,i+1}=1$ and $a_{i,i+m-1}=1$. But $a_{i,i+1}=(M_{i,i+2m-1})_{1,1}=1$ because this is the top left entry in $(\prod_{k=1}^{m} \eta(c_k))^2 = \begin{pmatrix} 1 & 0 \\ 0 & 1 \end{pmatrix}$. Moreover, $a_{i,i+m-1}=(M_{i,i+m-3})_{1,1}=1$ because
\begin{eqnarray*}
M_{i,i+m-3} & = & \prod_{k=i}^{i+m-3} \eta(c_k)
=  \left(\prod_{k=i}^{i+m-1} \eta(c_k)\right)\eta(c_{i+m-1})^{-1}\eta(c_{i+m-2})^{-1} \\
& = & - \eta(c_{i+m-1})^{-1}\eta(c_{i+m-2})^{-1} = 
\begin{pmatrix} 1 & -c_{i+m-2} \\ c_{i+m-1} & 1-c_mc_{i+m-2} \end{pmatrix},
\end{eqnarray*}
so we get the desired diagonal of 1's.

Finally, the periodicity follows immediately from the periodicity of the sequence $(c_k)_{k\in \mathbb{Z}}$.
\end{proof}

\begin{defin} \label{def:quiddity}
Let $R\subseteq \mathbb{C}$ be a subset.
A \emph{quiddity cycle} over $R$ is a sequence $(c_1,\ldots,c_m)\in R^m$ 
satisfying
\begin{equation}\label{etaid}
\prod_{k=1}^{m} \eta(c_k) = \begin{pmatrix}-1 & 0 \\ 0 & -1 \end{pmatrix}.
\end{equation}
To each quiddity cycle we have a {\em corresponding tame frieze pattern} by Proposition 
\ref{perck}.
\end{defin}

\begin{remar} \label{rem:qcreverse}
Let $(c_1,\ldots,c_m)\in R^m$ be a quiddity cycle.
\smallskip

\noindent
(1) The rotated cycle $(c_m,c_1,\ldots,c_{m-1})$ and
the reversed cycle $(c_m,c_{m-1},\ldots,c_1)$ are again quiddity cycles.

In fact, the first assertion easily follows from the fact that the negative identity matrix
commutes with every matrix. The second already appeared in \cite[Proposition 5.3\,(3)]{CH09}; for
completeness we include the argument. Let $\tau=\begin{pmatrix} 0 & 1 \\ 1 & 0 \end{pmatrix}$;
note that $\tau^2$ is the identity matrix and that 
for every $c\in R$ we have $\tau \eta(c) \tau = \begin{pmatrix} 0 & 1 \\ -1 & c \end{pmatrix} = 
\eta(c)^{-1}$. It follows that
\begin{eqnarray*}
\eta(c_m)\eta(c_{m-1})\ldots \eta(c_1) & = & (\tau\eta(c_m)^{-1}\tau)(\tau\eta(c_{m-1})^{-1}\tau)\ldots
(\tau\eta(c_1)^{-1}\tau) \\ 
& = & \tau(\eta(c_1)\ldots \eta(c_{m-1})\eta(c_m))^{-1} \tau 
 = \begin{pmatrix} -1 & 0 \\ 0 & -1 \end{pmatrix}
\end{eqnarray*}
(where in the last equation we have used that $(c_1,\ldots,c_m)$ is a quiddity cycle).
\smallskip

\noindent
(2) We have $\prod_{i=1}^m \eta(-c_i) = (-1)^{m+1} \begin{pmatrix} 1 & 0 \\ 0 & 1 \end{pmatrix}$.
In particular, if $m$ is even, $(-c_1,\ldots,-c_m)$ is again a quiddity cycle.

In fact, set $T=\begin{pmatrix} \ii & 0 \\ 0 & -\ii \end{pmatrix}$, where $\ii\in\mathbb{C}$ is the 
imaginary unit. A direct computation shows that 
$T\eta(c)T=\eta(-c)$ for all $c\in \mathbb{C}$. Moreover, since $T^2$ is the negative of the identity  
matrix, we get
\begin{eqnarray*} 
\prod_{k=1}^m \eta(-c_k) & = & \prod_{k=1}^m T\eta(c_k)T = (-1)^{m-1} T (\prod_{k=1}^m \eta(c_k)) T
 =  (-1)^m T^2 = (-1)^{m+1} \begin{pmatrix} 1 & 0 \\ 0 & 1 \end{pmatrix}.
\end{eqnarray*}
\end{remar}

\begin{examp}
\label{ex:quiddity}
From the definition of the matrices $\eta(c)$ it is clear that there is no
quiddity cycle of length $m=1$. A straightforward calculation yields 
$$\eta(c_1)\eta(c_2) = \begin{pmatrix} c_1c_2-1 & -c_1 \\ c_2 & -1 \end{pmatrix}
$$
from which it follows that the only quiddity cycle of length $m=2$ is $(0,0)$.
For $m=3$ we compute
$$ \eta(c_1)\eta(c_2)\eta(c_3) = 
\begin{pmatrix} c_1c_2c_3-c_3-c_1 & -c_1c_2+1 \\ c_2c_3-1 & -c_2 \end{pmatrix}
$$
and deduce that the only quiddity cycle of length $m=3$ is $(1,1,1)$. For $m=4$ 
we obtain
$$ \eta(c_1)\eta(c_2)\eta(c_3)\eta(c_4) = 
\begin{pmatrix} c_1c_2c_3c_4-c_1c_4-c_3c_4-c_1c_2+1 & -c_1c_2c_3+c_1+c_3 \\ 
c_2c_3c_4-c_4-c_2 & -c_2c_3+1 \end{pmatrix}.
$$
For this to become the negative identity matrix we get $c_2c_3=2$ (from the $(2,2)$-entry)
and then $c_1=c_3$ and $c_2=c_4$ from the off-diagonal entries; with these conditions
the $(1,1)$-entry becomes $-1$. Hence the quiddity
cycles of length $m=4$ are precisely $(c_1,2c_1^{-1},c_1,2c_1^{-1})$ with $c_1\neq 0$.
\end{examp}

\begin{remar} \label{rem:quiddity}
For fixed $n\in\NN$, there may be different periodic frieze patterns 
with period $m=n+3$ and of height $n$ but with the same quiddity cycle: the array
\[
\begin{array}{ccccccccccc}
 &  & \ddots  & &  &  &  &  &  & & \\
 &  &  &  & &  &  &  &  & & \\
1 & a & -1 & d & 1 & 0 &  &  &  & & \\
0 & 1 & 0 & -1 & 0 & 1 & 0 & & &  & \\
 & 0 & 1 & b & -1 & e & 1 & 0 &  &  & \\
 &  & 0 & 1 & 0 & -1 & 0 & 1 & 0 &  & \\
 &  & & 0 & 1 & c & -1 & f & 1 & 0 & \\
 &  &  &  & 0 & 1 & 0 & -1 & 0 & 1 & 0 \\
 &  &  &  &  &  &  &  & \ddots & &
\end{array}
\]
(repeated to both sides) is a frieze pattern for arbitrary $a,b,c,d,e,f$, and it is periodic with period 6.
On the other hand, a direct matrix calculation shows that 
$$\eta(a)\eta(0)\eta(b)\eta(0)\eta(c)\eta(0) = 
\begin{pmatrix} -1 & -a-b-c \\ 0 & -1
\end{pmatrix}.
$$
By Definition \ref{def:quiddity},
the sequence $(a,0,b,0,c,0)$ is a quiddity cycle if and only if $a+b+c=0$.
So for example if $a=1$, $b=1$, $c=-2$, then $(a,0,b,0,c,0)$ is a quiddity cycle and the above array is a frieze pattern for arbitrary $d,e,f$, but tame only if $a=e$, $d=c$, $b=f$.
\end{remar}

%%%%%%%%%%%%%%%%%%%%%%%%%%%%%%%%%%%%%%%%%%%%%%%%%%%%%%%%%%%%%%%%%%%%%%%%%%%%%%%%%%
%%%%%%%%%%%%%%%%%%%%%%%%%%%%%%%%%%%%%%%%%%%%%%%%%%%%%%%%%%%%%%%%%%%%%%%%%%%%%%%%%%

\section{Bounds and finiteness}
\label{sec:bounds}

When considering frieze patterns over a subset $R\subseteq \mathbb{C}$, one of the most
fundamental questions is whether for a given $n\in \mathbb{N}$ there are finitely or infinitely
many frieze patterns over $R$ of height $n$. We have seen above (e.g.\ in Remark \ref{rem:quiddity})
that if zeroes are allowed as entries then usually one will get infinitely many frieze patterns
and the situation is hard to control. Therefore, for the questions of finiteness addressed
in this section we shall restrict to non-zero frieze patterns.

If for a certain subset $R$ there are only finitely many non-zero frieze patterns over $R$ of any
height $n$ then the additional
question arises whether one can find a useful combinatorial model of these frieze patterns.

Both questions are very hard in general and so far have only received answers in very few
cases. In fact, for $R=\mathbb{N}$ the classic Conway-Coxeter results \cite{CC73}
yield a bijection
between frieze patterns over $\mathbb{N}$ of height $n$ and triangulations of the regular
$(n+3)$-gon. In particular, the frieze patterns over $\mathbb{N}$ are counted by the 
famous Catalan numbers.

This result has been extended by Fontaine \cite{F14} to the case $R=\mathbb{Z}\setminus\{0\}$,
here only few new frieze patterns appear; more precisely, if $n$ is even
then every non-zero frieze pattern of height $n$ 
over the integers is a Conway-Coxeter frieze pattern 
and if $n$ is odd then there are twice as many non-zero frieze
patterns of height $n$ over the integers as Conway-Coxeter frieze patterns, and the new ones are 
obtained by multiplying in the Conway-Coxeter frieze patterns every second row by $-1$.

In this section we will provide rather general criteria for when there are
only finitely many non-zero frieze patterns over a subset $R\subseteq \mathbb{C}$.

As mentioned above, a non-zero frieze pattern is uniquely determined by its quiddity 
cycle. Our approach here is to guarantee small entries and to give upper bounds for all 
entries in a quiddity cycle. For this we use the absolute value $|\cdot|$ of complex
numbers and its well-known properties, without further mentioning.
\bigskip

The following useful lemma is inspired by old results on continued fractions.

\begin{lemma}\label{bound2}
Let $c_1,\ldots,c_m,d,e \in \CC$ with
$|c_m|\ge 1$ and
$|c_1e-d|>|e|$; moreover suppose that
\[ \prod_{j=1}^m \eta(c_j) = \begin{pmatrix} d & * \\ e & * \end{pmatrix}. \]
(Here the $*$'s denote arbitrary entries, not necessarily the same.)
Then there exists an index $j\in\{2,\ldots,m-1\}$ with $|c_j|<2$.
\end{lemma}

\begin{remar}
One might wonder that the conclusion of the lemma only applies if $m\ge 3$. In fact,
for $m=1$ and $m=2$ the assumptions of the lemma can not be satisfied. For $m=1$ one would get 
$e=1$ and $d=c_1$ and then $|c_1e-d|=0$. For $m=2$ one has
$$\eta(c_1)\eta(c_2) = \begin{pmatrix} c_1c_2-1 & -c_1 \\ c_2 & -1 \end{pmatrix},
$$
thus $e=c_2$ and $d=c_1c_2-1$. But then the assumptions would yield 
$1\le |c_2|=|e| < |c_1e-d|= 1$, a contradiction.
\end{remar}

\begin{proof}
Let $a,b\in \CC$ with $|a|\ge |b|$ and $|c|\ge 2$. Then
\[ |ac-b|\ge |ac|-|b| = |a|(|c|-1)+|a|-|b| \ge |a|(|c|-1)\ge |a|. \]
From this inequality and
\begin{equation} \label{eq:eta}
\eta(c)\begin{pmatrix} a \\ b \end{pmatrix} = 
\begin{pmatrix} c & -1 \\ 1 & 0 \end{pmatrix} \begin{pmatrix} a \\ b \end{pmatrix} =
\begin{pmatrix} ac-b \\ a \end{pmatrix}, 
\end{equation}
we see that multiplying vectors in $\CC^2$ from the left with $\eta(c)$, where $|c|\ge 2$,
preserves the property that the absolute value of the first entry is greater or equal to the absolute value of the second entry.

Now
\[ \eta(c_m)\begin{pmatrix} 1 \\ 0 \end{pmatrix} = 
\begin{pmatrix} c_m & -1 \\ 1 & 0 \end{pmatrix} \begin{pmatrix} 1 \\ 0 \end{pmatrix} = 
\begin{pmatrix} c_m \\ 1 \end{pmatrix}.
\]
From this and the assumption on the shape of $\prod_{j=1}^m \eta(c_j)$ we get
\begin{eqnarray*}
\left(\prod_{j=2}^{m-1} \eta(c_j)\right) \begin{pmatrix} c_m \\ 1 \end{pmatrix}
=  \eta(c_1)^{-1} \begin{pmatrix} d & * \\ e & * \end{pmatrix} 
\begin{pmatrix} 1 \\ 0 \end{pmatrix} 
=  \begin{pmatrix} 0 & 1 \\ -1 & c_1 \end{pmatrix} \begin{pmatrix} d \\ e \end{pmatrix} 
=  \begin{pmatrix} e \\ c_1e-d \end{pmatrix}.
\end{eqnarray*}

Note that by assumption we have $|c_m|\ge 1$ and $|e|< |c_1e-d|$.
Thus $|c_2|,\ldots,|c_{m-1}|$ all greater or equal to $2$ would contradict
the property stated after Equation (\ref{eq:eta}).
\end{proof}

In a special case we can draw a stronger conclusion which will turn out to be important for some
later applications. Note that in particular the following corollary applies in the case of
quiddity cycles (cf.\ Definition \ref{def:quiddity}).

\begin{corol} \label{cor:small}
Let $(c_1,\ldots,c_m)\in \CC^m$ such that $\prod_{j=1}^m \eta(c_j)$ is a scalar multiple of the 
identity matrix.
Then there are two different indices $j,k\in\{1,\ldots,m\}$ with $|c_j|<2$ and $|c_k|<2$.
\end{corol}

\begin{remar} \label{rem:eta}
(i)
Note that by Definition \ref{def:etamatrix} the assumption of Corollary \ref{cor:small}
can only be satisfied for $m\ge 2$. For $m=2$, the only sequence $(c_1,c_2)$ with 
$\eta(c_1)\eta(c_2)=\begin{pmatrix} c_1c_2-1 & -c_1 \\ c_2 & -1 \end{pmatrix}$
a scalar multiple of the identity matrix is $(c_1,c_2)=(0,0)$.
\smallskip

(ii)
Moreover, the product $\prod_{j=1}^m \eta(c_j)$ is not the
zero matrix, since the matrices $\eta(c_j)$ have determinant 1.   
\end{remar}

\begin{proof}
The statement of the corollary clearly holds for the sequence $(0,0)$. So, according to Remark
\ref{rem:eta}\,(i) we can assume that $m\ge 3$.
\smallskip

A crucial initial observation is that scalar multiples of the identity matrix commute
with every matrix. This implies that if $(c_1,\ldots,c_m)\in \CC^m$ satisfies the assumption 
of the corollary, then also the rotated sequence $(c_m,c_1,\ldots,c_{m-1})$ does.
\smallskip

Suppose first that $|c_m|<1$. If also $|c_{m-1}|<1$, we are done. If $|c_{m-1}|\ge 1$ then we
consider the rotated sequence $(c_m,c_1,\ldots,c_{m-1})$. By the initial observation, 
this satisfies the assumptions 
of Lemma \ref{bound2} (with $e=0$, and $d\neq 0$, cf.\ Remark \ref{rem:eta}\,(ii)), 
hence we obtain an index $j\in \{1,\ldots,m-2\}$ with
$|c_j|<2$, and we are also done.

So suppose from now on that $|c_m|\ge 1$. Then we can, as before, apply Lemma \ref{bound2}
and obtain an index $j\in \{2,\ldots,m-1\}$ with $|c_j|<2$. We then consider the rotated
sequence $(c_j,c_{j+1},\ldots,c_m,c_1,\ldots,c_{j-2},c_{j-1})$. If
$|c_{j-1}|<1$ we are done (choose $k=j-1$). If $|c_{j-1}|\ge 1$ then we can again apply 
Lemma \ref{bound2} (by iterating the initial observation, the sequence 
$(c_j,c_{j+1},\ldots,c_m,c_1,\ldots,c_{j-2},c_{j-1})$) satisfies the assumptions, with $e=0$ and $d\neq 0$).
Thus we get an index $k\in \{j+1,\ldots,m,1,\ldots,j-2\}$
with $|c_k|<2$. This finishes the proof.
\end{proof}

\begin{remar}
Let us revisit again the classic case of Conway-Coxeter frieze patterns (or slightly more 
generally of frieze patterns over $\mathbb{Z}\setminus\{0\}$). Then Corollary 
\ref{cor:small} states that in the quiddity cycle of every such frieze pattern 
there are two entries equal to 1 (or $\pm 1$ for frieze patterns over $\mathbb{Z}\setminus\{0\}$).
Actually this existence of 1's in the quiddity cycle is the key for establishing in \cite{CC73}
an inductive argument for showing the Conway-Coxeter bijection between frieze patterns 
over $\mathbb{N}$ and triangulations of regular polygons.
\end{remar}

We now state our rather general criterion which gives, under certain conditions on the
subset $R\subseteq \mathbb{C}$, 
upper bounds for the absolute values of entries in the quiddity cycle of a frieze pattern
over $R\setminus\{0\}$.
The idea of the proof below is inspired by experiments in the Master thesis of D.~Azadi \cite{A16}.

\begin{theor} \label{prop:bound}
Let $R\subseteq \mathbb{C}$ be a subset such that 
$$M:=\inf\{|x|\,:\,x\in R\setminus\{0\} \}>0
$$
(i.e.\ there is a non-zero lower bound on the absolute values of non-zero elements in $R$).
Let $F$ be a frieze pattern over $R\setminus \{0\}$ with height $n\in \mathbb{N}$. Then every 
entry in the 
quiddity cycle of $F$ has absolute value at most $\frac{(n-1)+2M}{M^2}$.
\end{theor}

\begin{proof} We set $B:=\frac{(n-1)+2M}{M^2}$ for abbreviation; note that $B>0$ since $M>0$
by assumption.
 
Suppose for a contradiction that there is an element $x_1$ in the 
quiddity cycle of $F$ such that 
\begin{equation} \label{eq:x1}
|x_1|> B.
\end{equation}  
Then we consider two consecutive rows in the frieze pattern $F$, as in the following figure.
$$\begin{array}{cccccccc}
0 & 1 & x_1 & \ldots & x_n & 1 & 0 & \\
  & 0 & 1 & y_1 & \ldots & y_n & 1 & 0
\end{array}
$$  
We compare neighbouring entries and proceed inductively. By the defining rule for frieze patterns
we have $x_1y_1-x_2=1$, so 
$|y_1|\le \frac{1+|x_2|}{|x_1|}$
by the triangle inequality.
By definition of $M$ as infimum and by our assumption on $x_1$ (cf.\ Equation (\ref{eq:x1})) 
we can conclude that
\begin{equation} \label{eq:y1}
M\le |y_1| \le \frac{1+|x_2|}{|x_1|} < \frac{1+|x_2|}{B}.
\end{equation}
Since $B>0$, this implies that 
\begin{equation} \label{eq:x2}
|x_2| > MB-1.
\end{equation}
 
Now we go one step further. Again, by the defining rule of frieze patterns, we have 
$x_2y_2-x_3y_1=1$, so 
\begin{equation} 
|y_2|\le \frac{1+|x_3|\cdot |y_1|}{|x_2|}.
\end{equation}
Together with the definition 
of $M$ and Equation (\ref{eq:y1}) we obtain
$$
M  \le  |y_2| \le \frac{1+|x_3|\cdot |y_1|}{|x_2|} < \frac{B+|x_3| + |x_3|\cdot |x_2|}{B|x_2|} 
  =  \frac{1}{|x_2|} + \frac{|x_3|}{B} \left( \frac{1}{|x_2|}+1 \right).
$$
Now we use Equation (\ref{eq:x2}) and get
\begin{equation} \label{eq:y2}
M\le |y_2|  <  \frac{1}{MB-1} + \frac{|x_3|}{B} \cdot \frac{MB}{MB-1}
 = \frac{1}{MB-1}(1+M|x_3|).
\end{equation}
Solving this inequality for $|x_3|$ yields
\begin{equation} \label{eq:x3}
|x_3| > \frac{M(MB-1)-1}{M}.
\end{equation}
(Note that $MB-1 = \frac{n-1+2M}{M}-1 = \frac{(n-1)+M}{M} >0$ so the inequality sign is not reversed.)

\smallskip

Now we suppose inductively that we have already shown
\begin{equation} \label{eq:xk}
|x_k| > \frac{M(MB-1)-(k-2)}{M} \mbox{~~~~~~~for $k=2,\ldots, i+1$}
\end{equation}
and
\begin{equation} \label{eq:yk}
|y_k| < \frac{M}{M(MB-1)-(k-2)} (1+M|x_{k+1}|) \mbox{~~~~~~~for $k=2,\ldots, i$}.
\end{equation}

Note that Equations (\ref{eq:x2}) and (\ref{eq:x3}) are the special case of Equation  
(\ref{eq:xk}) for $k=2$ and $k=3$; moreover, Equation (\ref{eq:y2}) shows that Equation 
(\ref{eq:yk}) is valid for $k=2$.   
This means that the induction base for $i=2$ has been settled above.
\smallskip

For the induction step we now consider $y_{i+1}$ (where $i+1\le n$); 
by the defining rule for frieze patterns this has the 
form $y_{i+1} = \frac{1+x_{i+2}y_i}{x_{i+1}}$. Analogous to the arguments for the induction base
we get the following series of inequalities by using the triangle inequality and the induction
hypotheses for $|y_i|$ in Equation (\ref{eq:yk}) and for $|x_{i+1}|$ in Equation (\ref{eq:xk}),
respectively.
\begin{eqnarray*}
|y_{i+1}| & \le & \frac{1+|x_{i+2}|\cdot |y_i|}{|x_{i+1}|}  \\
& < & \frac{1}{|x_{i+1}|} + \frac{M|x_{i+2}|}{M(MB-1)-(i-2)} \left( \frac{1}{|x_{i+1}|} + M\right) \\
& < & \frac{M}{M(MB-1)-(i-1)} + \frac{M|x_{i+2}|}{M(MB-1)-(i-2)} 
\left( \frac{M}{M(MB-1)-(i-1)} + M\right) \\
& = & \frac{M}{M(MB-1)-(i-1)} \left( 1 + M |x_{i+2}| \right).
\end{eqnarray*}
This shows the induction step for Equation (\ref{eq:yk}). But using that $M\le |y_{i+1}|$
and then solving the above inequality for $|x_{i+2}|$ yields
$$|x_{i+2}| > \frac{M(MB-1)-i}{M}
$$
and this gives the induction step for Equation (\ref{eq:xk}) as well.
(Note again that multiplication does not reverse the inequality sign since 
$M(MB-1)-(i-1) = n-1+M-(i-1) = n-i+M >0$ since $i\le n-1$.)
\smallskip

We are considering a frieze pattern $F$ of height $n$. Thus, $x_{n+1}=1$ and eventually we get from
Equation (\ref{eq:yk}) that
\begin{equation} \label{eq:yn}
|y_n| <  \frac{M}{M(MB-1)-(n-2)} \left( 1+M \right).
\end{equation}
Using the definition of $B$ we compute that 
$$M(MB-1) = M\left( \frac{(n-1)+2M}{M} -1 \right) = n-1+M.
$$
Together with Equation (\ref{eq:yn}) this yields $|y_n| < M$, contradicting the definition 
of $M$ as infimum and the fact that our frieze pattern $F$ has non-zero entries by assumption.
\end{proof}

\begin{remar}
For the classic case of Conway-Coxeter frieze patterns over $\mathbb{N}$ (or frieze
patterns over $\mathbb{Z}\setminus\{0\}$) the above theorem
yields that every entry in the quiddity cycle of such a frieze pattern of height $n$ has 
absolute value $\le n+1$. This is not hard to deduce from the bijection between frieze 
patterns over $\mathbb{N}$ 
and triangulations because the entries in the quiddity cycle are given by the 
numbers of triangles attached to the vertices.
However, the above proof does not refer to this bijection.
\end{remar}

As one of our main applications of Theorem \ref{prop:bound} we can deduce for a large class 
of subsets $R\subseteq \mathbb{C}$ that there are only finitely many frieze patterns 
over $R\setminus \{0\}$ for any given height.

\begin{corol} \label{cor:discrete}
Let $R\subseteq \CC$ be a discrete subset (i.e.\ $R$ has no accumulation point).
Then for each $n\in \mathbb{N}$ there are only finitely
many frieze patterns over $R\setminus \{0\}$ of height $n$.
\end{corol}

\begin{proof}
The subset $R$ is discrete, in particular the origin is not an accumulation point for $R$, so  
$R$ satisfies the assumption of Theorem \ref{prop:bound}. Furthermore, again by discreteness,
each closed disk of $\mathbb{C}$ contains only finitely many elements of $R$. Together with 
Theorem \ref{prop:bound} this implies that for fixed
height $n$ there are only finitely many possible elements of $R$ which can appear in a quiddity cycle
of a frieze pattern of height $n$. But any non-zero frieze pattern is uniquely determined by its quiddity cycle.
So the claim follows.
\end{proof}

For dealing with a second main application of our finiteness criterion Theorem \ref{prop:bound} 
we first state an observation which easily follows from a construction we will present in the
next section.

\begin{propo} \label{prop:units}
Let $R\subseteq \CC$ be a subring containing infinitely many divisors of $2$ (i.e.\ elements $t\in R$ s.t.\ 
$\frac{2}{t}\in R$).
Then for each $n\in \mathbb{N}$
there are infinitely many frieze patterns of height $n$ over $R\setminus \{0\}$.
\end{propo}

\begin{proof}
Let $t\in R$ be an element such that $\frac{2}{t}\in R$. 

Then  
the sequence $(t,\frac{2}{t},t,\frac{2}{t})$ is the quiddity cycle of a frieze pattern of height 1
over $R\setminus \{0\}$, thus proving the claim for $n=1$.

Now fix some $n\ge 2$. For any $t\in R$ with $\frac{2}{t}\in R$, start with the quiddity cycle 
$(t,\frac{2}{t},t,\frac{2}{t})$.
Then use the rule from Proposition \ref{eta_rule} to insert $n-1$ times a 1 between 
the first and second entry in the quiddity cycle. By Proposition \ref{eta_rule} this yields a new
quiddity cycle of length $n-3$ and this has the form
$$(t+n-1,1,2,\ldots,2,1+\frac{2}{t},t,\frac{2}{t}).
$$
One checks that the fundamental domain of the corresponding frieze pattern has the following form:
$$\begin{array}{cccccccccc}
1 & t+n-1 & t+n-2 & t+n-3 & \ldots & \ldots & \ldots & t+1 & t & 1 \\
& & & & & & & & & \\
& 1 & 1 & 1 & 1 & \ldots & \ldots & 1 & 1 & \frac{2}{t} \\
& & & & & & & & & \\
& & 1 & 2 & 3 & 4 & \ldots & n-2 & n-1 & 1+\frac{2(n-1)}{t} \\
& & & & & & & & & \\
& & & 1 & 2 & 3 & \ldots & n-3 & n-2 & 1+\frac{2(n-2)}{t} \\
& & & & & & & & & \\
& & & & \ddots & \ddots & & \vdots & n-3 & 1+\frac{2(n-3)}{t} \\
& & & & & & & & & \\
& & & & & \ddots & \ddots & \vdots & \vdots & \vdots \\
& & & & & & & & & \\
& & & & & & 1 & 2 & 3 & 1+\frac{6}{t} \\
& & & & & & & & & \\
& & & & & & & 1 & 2 & 1+\frac{4}{t} \\
& & & & & & & & & \\
& & & & & & & & 1 & 1+\frac{2}{t} \\
& & & & & & & & & \\
& & & & & & & & & 1 \\
\end{array} 
$$
Note that all entries in this frieze pattern are in $R$ since $R$ is a subring and $\frac{2}{t}\in R$.

Moreover, only for finitely many $t$ some entry in the frieze pattern becomes zero, namely
for $t\in \{-1,-2,\ldots,-(n-1)\}$ and for $t\in \{-2,-4,-6,\ldots,-2(n-1)\}$. 

By assumption, there are infinitely many possible such $t$. This implies that there are
infinitely many frieze patterns of height $n$ over $R\setminus \{0\}$, as claimed.
\end{proof}

As a second main application of our results 
in this section we can provide an answer to a question posed
by Fontaine in \cite{F14}. At the end of his paper, Fontaine considered frieze patterns over rings 
$\mathbb{Z}[\zeta]$ where $\zeta\in \mathbb{C}$ is a root of unity and he stated the conjecture
that over each of these rings there are only finitely many non-zero frieze patterns of any given height
(see \cite[Conjecture 6.2]{F14}).

This conjecture has a negative answer in general for any height $n$;
only for few such rings one indeed gets finitely many 
frieze patterns. The following result gives a complete answer.

\begin{corol} \label{cor:zeta_d}
For $d\in \mathbb{N}$, let $\zeta_d\in \mathbb{C}$ be a primitive $d$-th root of unity. The
following statements are equivalent:
\begin{enumerate}
\item[{(i)}] For every $n\in \mathbb{N}$, the number of non-zero frieze patterns of height $n$ 
over $\mathbb{Z}[\zeta_d]$ is finite.
\item[{(ii)}] The set $\mathbb{Z}[\zeta_d]$ is discrete in $\mathbb{C}$.
\item[{(iii)}] The group of units of $\mathbb{Z}[\zeta_d]$ is finite. 
\item[{(iv)}] $d\in \{1,2,3,4,6\}$.
\end{enumerate}
\end{corol}

\begin{proof}
We show the implications $(iv)\Rightarrow (ii) \Rightarrow (i) \Rightarrow (iii) \Rightarrow (iv)$. 

Let us first suppose that $(iv)$ holds.
For the specific values given in $(iv)$ 
we have $\mathbb{Z}[\zeta_1]=\mathbb{Z}= \mathbb{Z}[\zeta_{2}]$, the
integers, 
$\mathbb{Z}[\zeta_4]=\mathbb{Z}[\ii]$, the Gaussian integers, and 
$\mathbb{Z}[\zeta_3]=\mathbb{Z}[\frac{1+\ii\sqrt{3}}{2}]= \mathbb{Z}[\zeta_{6}]$, the Eisenstein integers.
These rings are easily seen 
to be discrete subsets of $\mathbb{C}$, so $(ii)$ holds. 

Let us now suppose that $(ii)$ holds.
Then $(i)$ follows directly from Corollary \ref{cor:discrete}.

Now suppose that $(i)$ holds. Then it follows from Proposition \ref{prop:units} that the ring
$\mathbb{Z}[\zeta_d]$ has only finitely many units, thus $(iii)$ holds.  

We now suppose that $(iii)$ holds. 
The ring $\mathbb{Z}[\zeta_d]$ is the ring of integers of the
cyclotomic number field $\mathbb{Q}(\zeta_d)$ (see e.g.\ \cite[Proposition 10.2]{Neukirch}).
According to Dirichlet's unit theorem,  
the rank of the group of units of $\mathbb{Z}[\zeta_d]$ has the form $r_1+r_2-1$ where $r_1$ and $r_2$
are the numbers of real embeddings and pairs of complex embeddings of $\mathbb{Q}(\zeta_d)$,
respectively. Moreover,
$r_1+2r_2 = |\mathbb{Q}(\zeta_d):\mathbb{Q}|=\varphi(d)$, where $\varphi$ denotes Euler's totient
function.

By assumption $(iii)$, the rank of the group of units must be 0, thus
$r_1+r_2=1$. The are only two possibilities: we can have $r_1=1$ and $r_2=0$, and then 
$\varphi(d)=r_1+2r_2=1$; or we have $r_1=0$ and $r_2=1$, and then $\varphi(d)=2$. It is an easy exercise 
on Euler's totient function to show that $\varphi(d)=1$ if and only if $d\in \{1,2\}$ and
$\varphi(d)=2$ if and only if $d\in \{3,4,6\}$.
So statement $(iv)$ holds.
\end{proof}

\begin{remar}
For $d\in \{1,2\}$, the finitely many non-zero frieze patterns over the integers have been classified 
by Fontaine \cite{F14}. In addition to the classic Conway-Coxeter frieze patterns 
over $\mathbb{N}$, new frieze patterns over $\mathbb{Z}\setminus\{0\}$ only appear when the
height $n$ is odd, and then one gets twice as many frieze patterns as Conway-Coxeter frieze patterns.
\smallskip

\noindent
For $d=4$ and $d\in \{3,6\}$ the situation is more complex. In addition to the Conway-Coxeter
frieze patterns, there are plenty of new non-zero
frieze patterns over the Gaussian integers $\mathbb{Z}[\ii]$ and the Eisenstein integers $\mathbb{Z}[\frac{1+\ii\sqrt{3}}{2}]$.
A major complication compared to the classic
Conway-Coxeter frieze patterns (or frieze patterns over $\mathbb{Z}\setminus\{0\}$) is that in the quiddity
cycles one does not necessarily have an entry equal to 1 (or $\pm 1$). For instance, Example \ref{ex:friezes}\,(2)
gives an example of a non-zero frieze pattern over the Gaussian integers without 1's in the quiddity cycle.
However, Corollary \ref{cor:small} yields that in each non-zero frieze pattern over $\mathbb{Z}[\ii]$ one has two entries
of absolute value less than 2, i.e.\ entries in $\{\pm 1, \pm \ii, \pm (1+\ii), \pm (1-\ii)\}$. 

It seems to be a subtle problem to classify all non-zero frieze patterns over the Gaussian integers and the
Eisenstein integers. With the help of a computer we have calculated the numbers of such frieze patterns for small
heights. For $n\in \mathbb{N}$ denote by $G_n$ and $E_n$ the numbers of non-zero frieze patterns of height $n$
over the Gaussian integers and Eisenstein integers, respectively. 
Moreover, there is an action of the dihedral group $D_{n+3}$ on the set of non-zero frieze patterns of 
height $n$, coming from rotating and reflecting the corresponding quiddity cycles, see Remark \ref{rem:qcreverse}.
We have also computed the numbers of non-zero frieze patterns 
modulo this action, and denote them by 
$G'_n$ and $E'_n$. 

Then we summarize our computer calculations in the following table.  
\begin{center}
\begin{tabular}{|c||c|c|c|c|}
\hline
$n$ & 1 & 2 & 3 & 4 \\
\hline\hline
$G_n$ & 12 & 55 & 668 & 4368 \\
\hline
$G'_n$ & 6 & 7 & 81 & 323 \\
\hline
$E_n$ & 12 & 75 & 1062 & 8526 \\
\hline
$E'_n$ & 6 & 10 & 127 & 628 \\
\hline
\end{tabular}
\end{center}
\end{remar}

%%%%%%%%%%%%%%%%%%%%%%%%%%%%%%%%%%%%%%%%%%%%%%%%%%%%%%%%%%%
%%%%%%%%%%%%%%%%%%%%%%%%%%%%%%%%%%%%%%%%%%%%%%%%%%%%%%%%%%%
\section{Transformations and rules}
\label{sec:transform}

Similar to the classic Conway-Coxeter theory we would like to 
classify frieze patterns over some subset $R\subseteq \mathbb{C}$ 
via their quiddity cycles. More precisely, we want to prove that every quiddity cycle admits 
transformations leading to a shorter cycle.

All results in this section hold for matrices over an arbitrary commutative ring $R$, so we take a slightly 
more general viewpoint here, setting $\eta(c)=\begin{pmatrix} c & -1 \\ 1 & 0 \end{pmatrix}$
for $c\in R$ as in Definition \ref{def:etamatrix}.

We first recall some easy but very useful transformations, showing that in a quiddity cycle one can 
insert/delete $\pm 1$'s.

\begin{propo}[See also {\cite[Lemma 5.2]{CH09}}]
\label{eta_rule}
For all $a,b\in R$ we have
\begin{enumerate}
\item[{(a)}] 
$\eta(a)\eta(b) = \eta(a+1)\eta(1)\eta(b+1).$
\item[{(b)}] $\eta(a)\eta(b) = -\eta(a-1)\eta(-1)\eta(b-1).$
\end{enumerate}
\end{propo}

\begin{proof} Both formulas can easily be verified by direct calculations.
\end{proof}

The following lemma contains some of the (apparently) most useful rules concerning products of matrices of 
the form $\eta(a)$, $a\in\CC$. We will need them later in Section \ref{sec:Zfriezes}
for the main result about quiddity cycles over $\ZZ$.

\begin{lemma}\label{auvb}
Let $a,u,v,b,\lambda\in R$ and assume that $uv-1$ and $\lambda$ are invertible in $R$. Then we have:
\begin{eqnarray}
\label{eqauvb}\quad\quad \eta(a)\eta(u)\eta(v)\eta(b) &=& \eta\left(a+\frac{1-v}{uv-1}\right)\eta(uv-1)\eta\left(b+\frac{1-u}{uv-1}\right) 
\\
\quad\quad \eta(a)\eta(u)\eta(v)\eta(b) &=& \eta\left(a+\frac{(\frac{1}{\lambda}-1)v}{uv-1}\right)
\eta(\lambda u)\eta\left(\frac{v}{\lambda}\right)\eta\left(b+\frac{(\lambda-1)u}{uv-1}\right)
\\
\label{eqa0b}\quad\quad \eta(a)\eta(0)\eta(b) &=& \eta(a+u)\eta(0)\eta(b-u) = -\eta(a+b)
\end{eqnarray}
\end{lemma}
\begin{proof}
Each formula can be verified by straightforward (though slightly tedious) matrix computations.
We leave the details to the reader.
\end{proof}

Finally, we mention the following very useful general transformation. It is not immediately applied in the present paper but
it is the key to reductions for quiddity cycles over several subsets of $\CC$.  

\begin{lemma}\label{fabb}
Let $a,b\in R$ and  assume that $u,z$ are invertible in $R$. Then
\[
\begin{pmatrix} \frac{1}{z} & 0 \\ 0 & z \end{pmatrix} \eta(a)\eta(u)\eta(b) \begin{pmatrix} 
z & 0 \\ 0 & \frac{1}{z} \end{pmatrix} =
\eta\left(\frac{a}{z^{2}} -\frac{\frac{1}{z^2}-1}{u}\right) \eta(u) 
\eta\left(z^2 b-\frac{z^2-1}{u}\right).
\]
\end{lemma}
\begin{proof}
The formula can easily be verified by direct computation.
\end{proof}

%%%%%%%%%%%%%%%%%%%%%%%%%%%%%%%%%%%%%%%%%%%%%%%%%%%%%%%%%%%%%%%%%%%%%%%%%%%%%%%%%%
%%%%%%%%%%%%%%%%%%%%%%%%%%%%%%%%%%%%%%%%%%%%%%%%%%%%%%%%%%%%%%%%%%%%%%%%%%%%%%%%%%
\section{Frieze patterns as specializations of clusters algebras}

In this section we consider the close connection between frieze patterns and Fomin and 
Zelevinsky's cluster algebras in Dynkin type $A$. If you place indeterminates $x_1,\ldots,x_n$
on a row of a frieze pattern of height $n$ (say on the positions labelled 
$c_{0,2},c_{0,3},\ldots,c_{0,n+1}$ in Definition \ref{def:frieze}) and still follow the 
local condition $ad-bc=1$, then one obtains
a frieze pattern over the rational function field $\mathbb{Q}(x_1,\ldots,x_n)$ with entries
being the cluster variables of a cluster algebra of Dynkin type $A_n$. See Figure \ref{fig:A2} for
the case $n=2$.

\begin{figure} 
 $$   \begin{array}{ccccccccccc}
 &  & \ddots &  &  &  &  &  & &  & \\
 &  &  & &  &  &  &  &  & & \\
 0 & 1 & x_1 & x_2 & 1 & 0 &  &  &  &  & \\
& 0 & 1 & \frac{1+x_2}{x_1} & \frac{1+x_1+x_2}{x_1x_2} & 1 & 0 &  & & & \\
 & & 0 & 1 & \frac{1+x_1}{x_2} & x_1 & 1 & 0 & & & \\
 & &  & 0 & 1 & x_2 & \frac{1+x_2}{x_1} & 1 & 0 & & \\
 & &  &  & 0 & 1 & \frac{1+x_1+x_2}{x_1x_2} & \frac{1+x_1}{x_2} & 1 & 0 & \\
 &  &  &  &  &  &  &  & \ddots &  &
\end{array}
$$
\caption{The frieze pattern corresponding to the cluster algebra of Dynkin type $A_2$.}
\label{fig:A2}
\end{figure}

These frieze patterns corresponding to cluster algebras of Dynkin type $A_n$
are periodic of period $n+3$ (cf.
Proposition \ref{perck}), and since they are tame, they have an additional glide symmetry. 
In other words, a fundamental domain
for the entries in these frieze patterns (containing each cluster variable once) is given
by a triangular shape.
\smallskip

Note that the pairs of indices of the elements in this fundamental region are in bijection 
with pairs of different numbers from $\{0,\ldots,n+2\}$, and these are in bijection with the edges 
and the diagonals of a regular $(n+3)$-gon $\mathcal{P}$ (with vertices labelled consecutively).

It is well known that there is an alternative description of frieze patterns with such a glide symmetry.
Namely, such a frieze pattern can be seen as an assignment of numbers, called labels, to the edges and the
diagonals of $\mathcal{P}$ such that the labels of edges are 1 and such that for each pair of crossing
diagonals $(i,j)$ and $(k,\ell)$ the Ptolemy condition 
$$c_{i,j}c_{k,\ell} = c_{i,k}c_{j,\ell} + c_{i,\ell}c_{j,k}
$$
is satisfied; see Figure \ref{fig:Ptolemy}.

\begin{figure}
\[
  \begin{tikzpicture}[auto]
    \node[name=s, shape=regular polygon, regular polygon sides=20, minimum size=5cm, draw] {}; 
    \draw[thick, dotted] (s.corner 5) to node[very near start,below=24pt] {$c_{i,k}$} (s.corner 15);
    \draw[shift=(s.corner 5)] node[left] {$i$};
    \draw[shift=(s.corner 2)] node[below=17pt] {$c_{i,\ell}$};
    \draw[shift=(s.corner 12)] node[above=14pt] {$c_{j,k}$};
    \draw[shift=(s.corner 17)] node[left=11pt] {$c_{j,\ell}$};
    \draw[shift=(s.corner 2)] node[below=50pt] {$c_{i,j}$};
    \draw[shift=(s.corner 11)] node[above=40pt] {$c_{k,\ell}$};
    \draw[shift=(s.corner 15)] node[right] {$j$};
    \draw[thick, dotted] (s.corner 9) to (s.corner 19);
    \draw[shift=(s.corner 9)] node[left] {$k$};
    \draw[shift=(s.corner 19)] node[right] {$\ell$};
    \draw[thick] (s.corner 5) to (s.corner 9);
    \draw[thick] (s.corner 5) to (s.corner 19);
    \draw[thick] (s.corner 15) to (s.corner 9);
    \draw[thick] (s.corner 15) to (s.corner 19);
  \end{tikzpicture} 
\]
\caption{The Ptolemy condition.}
\label{fig:Ptolemy}
\end{figure}
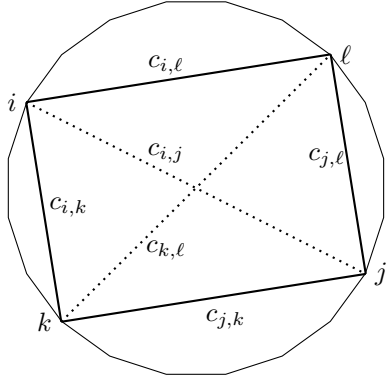

Any maximal set of pairwise non-crossing diagonals of $\mathcal{P}$ is called a {\em cluster}.
Note that the diagonals of a cluster form a triangulation of the polygon $\mathcal{P}$. 
If $\mathcal{F}$ is a tame frieze pattern of height $n$ we also call any set of entries of $\mathcal{F}$
placed at positions corresponding to a cluster of $\mathcal{P}$ a {\em cluster of $\mathcal{F}$}.

This notion is perfectly in line with the fundamental notion of cluster in the cluster algebras 
of Dynkin type $A$. In fact, a subset of cluster variables forms a cluster of the cluster algebra precisely
when the positions in the frieze pattern yield a triangulation, i.e.\ a cluster of $\mathcal{P}$
(and mutation of clusters corresponds to flipping diagonals in the triangulation).

A deep theorem on cluster algebras, the Laurent phenomenon, states that if you start with placing 
indeterminates $x_1,\ldots,x_n$ on the positions of a cluster then all cluster variables are
rational functions with denominator a monomial $x_1^{a_1}\ldots x_n^{a_n}$; see for instance
\cite[Theorem 3.3.1]{FWZ}.

We want to study what happens when the indeterminates, placed on the positions of a cluster,
are specialized to values of $\mathbb{C}$.  
From the above mentioned Laurent phenomenon it is clear that one can not specialize an indeterminate
to zero (because this resulted in undefined denominators).

We call a tame frieze pattern over $\mathbb{C}$ of height $n$ a {\em specialization} of the cluster 
algebra of Dynkin type $A_n$ if the frieze pattern contains a cluster with only non-zero entries.
\medskip

The main result of this section will describe the frieze patterns over $\mathbb{C}$ which are
specializations of the cluster algebras of Dynkin type $A$.
For proving this result we need some preparation.

\begin{propo}\label{cor:t}
Suppose $m\in \mathbb{N}$ is even.
Let $\underline{c}=(c_1,\ldots,c_m)\in \CC^m$ be a quiddity cycle, and $0\ne t\in\CC$. Then the following
hold:
\begin{enumerate}
\item[{(a)}] $\underline{c'}=(c'_1,\ldots,c'_m)=(t c_1, t^{-1} c_2, t c_3, \ldots, t^{-1} c_m)$ 
is a quiddity cycle as well.
\item[{(b)}] The 
frieze patterns
$\mathcal{F}$ and $\mathcal{F}'$ corresponding to $\underline{c}$ and $\underline{c'}$ have their zero
entries at exactly the same positions.
\end{enumerate}
\end{propo}

\begin{proof}
(a) For $t\in \mathbb{C}\setminus \{0\}$ we consider the complex matrix
\[ T := \begin{pmatrix} \sqrt{t} & 0 \\ 0 & \frac{1}{\sqrt{t}} \end{pmatrix}. \]
Then a simple computation yields that for every $c\in \CC$ we have
\begin{equation}\label{TT}
\eta(tc) = T \eta(c) T, \quad \eta\left(t^{-1}c\right) = T^{-1} \eta(c) T^{-1}.
\end{equation}
This implies that 
\begin{equation*}
\eta(tc_1)\eta(t^{-1}c_2)\ldots \eta(tc_{m-1})\eta(t^{-1}c_m) = T
\eta(c_1)\eta(c_2) \ldots \eta(c_{m-1})\eta(c_m) T^{-1} = 
\begin{pmatrix} -1 & 0 \\ 0 & -1 \end{pmatrix} 
\end{equation*}
where for the last equality we have used the assumption that $\underline{c}$ is a quiddity cycle
(cf.\ Definition \ref{def:quiddity}). Thus, $\underline{c}'$ is again a quiddity cycle.
\smallskip

\noindent
(b) We denote the entries in the frieze patterns $\mathcal{F}$ and $\mathcal{F}'$ by 
$c_{i,j}$ and $c'_{i,j}$, respectively. According to Proposition \ref{perck} the entries in the
frieze patterns are given by the $(1,1)$-entry in some product of $\eta$-matrices. More precisely,
we get  by using Equation (\ref{TT}) that for all $i\le j-2$:
$$c'_{i,j} = (\prod_{k=i}^{j-2} \eta(c'_k))_{1,1} = (T^{\delta} (\prod_{k=i}^{j-2} \eta(c_k))T^{\epsilon})_{1,1}
$$
with some $\delta,\epsilon\in \{1,-1\}$. But it is straightforward to check that for every complex
$2\times 2$-matrix $A$ we have 
$$ (TAT)_{1,1} = t A_{1,1}, \quad (TAT^{-1})_{1,1} = A_{1,1},
\quad(T^{-1}AT)_{1,1} = A_{1,1}, \quad(T^{-1}AT^{-1})_{1,1} = t^{-1} A_{1,1}.
$$
So the entries in the frieze patterns corresponding to $\underline{c}$ and 
$\underline{c}'$ only differ up to (possibly) multiplication with $t$ or $t^{-1}$.
In particular, the non-zero entries appear at the same positions.
\end{proof}

\begin{lemma}\label{lem:00111}
Suppose that $m\in \mathbb{N}$ is odd.
Let $\underline{c}=(c_1,\ldots,c_m)\in \CC^m$ be a quiddity cycle, and assume
that for all $i=1,\ldots,m$ we have $c_i c_{i+1}=1$ or $c_i=c_{i+1}=0$.
Then $\underline{c}=(1,\ldots,1)$ and $m\equiv 3\:(\md 6)$.
\end{lemma}

\begin{proof}
We claim that $c_1\neq 0$. In fact, if $c_1=0$ then it follows directly from the assumption that $c_1=c_2=\ldots=c_m=0$, i.e.
$\underline{c}=(0,\ldots,0)$. However, one easily computes that
$\eta(0)\eta(0)$ is the negative of the identity matrix. This implies that 
$(0,\ldots,0)\in \mathbb{C}^m$ is a quiddity cycle if and only if $m\equiv 2\:(\md 4)$.
Since $m$ is odd by assumption, we get a contradiction.

So we assume now that $c_1\ne 0$. But then $c_2=c_1^{-1}$ by the assumption, and 
inductively $\underline{c} = (c_1,c_1^{-1},c_1,\ldots,c_1^{-1},c_1)$ (use that $m$ is odd).
By assumption, $\underline{c}$ is a quiddity cycle, i.e.\ $\prod_{k=1}^m \eta(c_k)$ is
minus the identity matrix (cf.\ Definition \ref{def:quiddity}).  
We claim that this can only happen if $m\equiv 3\:(\md 6)$ and that in this case $c_1=1$, i.e.
$\underline{c}=(1,\ldots,1)$. In fact, one easily computes
that $\eta(c_1)\eta(c_1^{-1})\eta(c_1)\eta(c_1^{-1})\eta(c_1)\eta(c_1^{-1})$ equals the 
identity matrix. So it suffices to consider the cases $m\in \{1,3,5\}$.
Clearly, this can not
happen for $m=1$. For $m=3$ we have 
$$\eta(c_1)\eta(c_1^{-1})\eta(c_1) = \begin{pmatrix} -c_1 & 0 \\ 0 & -c_1^{-1} \end{pmatrix}
$$
and this implies $c_1=1$. For $m=5$ we compute
$$(\eta(c_1)\eta(c_1^{-1})^2\eta(c_1) = \begin{pmatrix} 0 & 1 \\ -1 & c_1^{-1} \end{pmatrix}
$$
which is clearly not a scalar multiple of the identity matrix.
\end{proof}

\begin{figure}
\begin{center}
\begin{minipage}[b]{0.4\textwidth}
\vspace{20pt}
\includegraphics[width=\textwidth]{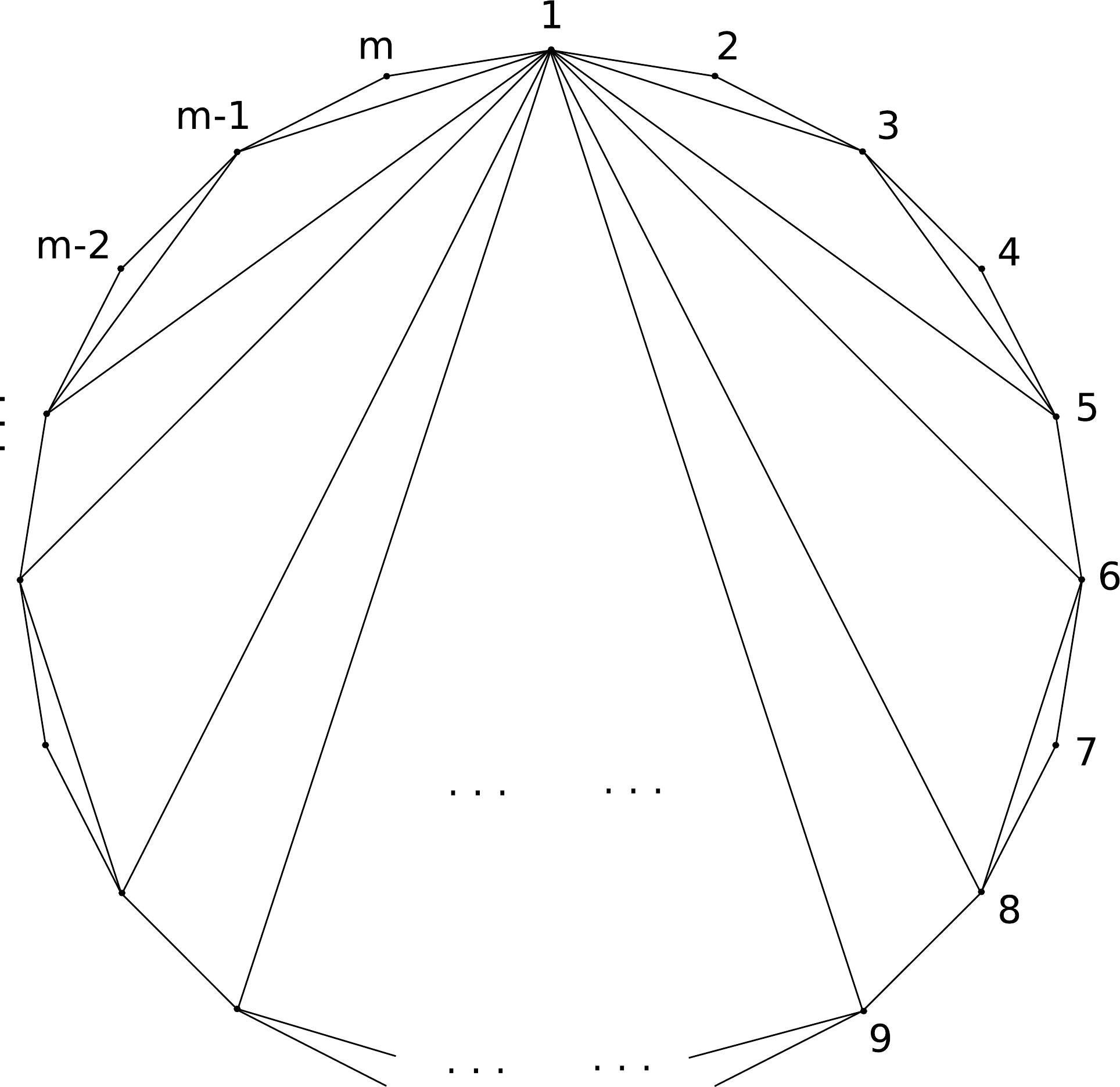}
\end{minipage}
\end{center}
\caption{The cluster appearing in Proposition \ref{prop:111} \label{p111}}
\end{figure}

\begin{propo}\label{prop:111}
Let $\underline{c}=(1,\ldots,1)\in \CC^m$ be a quiddity cycle and let $\Fc$ be the corresponding 
tame frieze pattern. Then there exists a cluster of $\Fc$ without zero entry.
\end{propo}

\begin{proof}
We have seen in the proof of Lemma \ref{lem:00111} that if $(1,\ldots,1)$ is a quiddity cycle then 
$m\equiv 3 \:(\md 6)$. Let $c_{i,j}$ denote the entries in the corresponding frieze pattern.
From Proposition \ref{perck} we know that these entries appear as the top left entry in 
matrices of the form $\eta(1)^{k}$ for $1\le k\le 6$
(note that since $\eta(1)^6$ is the identity matrix, the exponents can be reduced modulo 6).
Then an easy computation yields that for every $1\le i<j\le m$ we have
\[ c_{i,j} = \begin{cases}
1 & \text{if }j-i \equiv 1 \:(\md  6) \text{ or } j-i \equiv 2 \:(\md  6), \\
0 & \text{if }j-i \equiv 0 \:(\md  6) \text{ or } j-i \equiv 3 \:(\md  6), \\
-1 & \text{if }j-i \equiv 4 \:(\md  6) \text{ or } j-i \equiv 5 \:(\md  6).
\end{cases} \]
Thus it suffices to find a triangulation $C$ of the regular $m$-gon 
such that $j-i \not\equiv 0 \:(\md  3)$ for all diagonals $(i,j)\in C$. For example,
writing $m=6\ell +3$ we can take 
$$
C= \{(1,3)\} \cup \{(1,3k-1),(1,3k)\,|\,2\le k\le 2\ell\}\cup \{(3k,3k+2)\,|\,1\le k\le 2\ell\}
\cup \{(1,6\ell+2)\},
$$
see also Figure \ref{p111}.
\end{proof}

The following theorem is the main result of this section
and explains the connection between tame frieze patterns and specializations of cluster variables of 
Dynkin type $A$.

\begin{theor}\label{thm:clusters}
For $m\ge 4$, let $\underline{c}=(c_1,\ldots,c_m)\in \CC^m$ be a quiddity cycle with at least one 
non-zero entry and let $\Fc$ be the corresponding tame
frieze pattern. Then there exists a cluster of $\mathcal{F}$ without zero entry.
In particular, every such quiddity cycle defines a specialization of the variables of a cluster algebra 
of Dynkin type $A$ to complex numbers.
\end{theor}

\begin{proof}
We proceed by induction over the length $m$ of the quiddity cycle.

For $m=4$ every quiddity cycle has the form $(c_1,\frac{2}{c_1},c_1,\frac{2}{c_1})$ with $c_1\neq 0$
(cf.\ Example \ref{ex:quiddity}). The clusters of the corresponding quadrangle consist of only one diagonal, 
and this is labelled by $c_1$ or by $\frac{2}{c_1}$; since both are non-zero, the claim holds for $m=4$.

Suppose now that $m>4$.

Case 1: Assume first that there is a $k$ with $c_k=1$. By Proposition \ref{prop:111} we can assume 
that $\underline{c}$ contains an entry not equal to 1. So we can choose $k$ maximal such that
$c_k=1$ and $c_{k+1}\neq 1$.

From Corollary \ref{eta_rule}\,(a) we know
that $\underline{c}'=(c_1,\ldots,c_{k-2},c_{k-1}-1,c_{k+1}-1,c_{k+2},\ldots,c_m)$ is again a quiddity cycle.
By maximality of $k$, it has a non-zero entry, i.e.\ it satisfies the assumption of the theorem.
Now by induction hypothesis we find a cluster without zero for the frieze pattern $\mathcal{F}'$ 
corresponding to the quiddity 
cycle $\underline{c}'$ of length $m-1$. In other words, there is a triangulation of the $(m-1)$-gon
obtained from the $m$-gon by removing the vertex $k$ such that each diagonal has a non-zero label.
The diagonal $(k-1,k+1)$ in the $m$-gon
has label $c_k=1\neq 0$. So the cluster for $\mathcal{F}'$
may be extended to a cluster without zero of the original frieze pattern $\mathcal{F}$
by adding the diagonal $(k-1,k+1)$.
\smallskip

Case 2: Suppose that there is no $1$ in the quiddity cycle and that $m$ is even.
Then we choose any $k$ with $c_k\ne 0$.
Applying Proposition \ref{cor:t} with $t=c_k$ (if $k$ is even) or $t=c_k^{-1}$ (if $k$ is odd)
we obtain a quiddity cycle of the same length, with the same clusters without zero and with a $1$ 
at position $k$. We are thus in Case 1 and hence the claim follows.
\smallskip

Case 3: Finally, suppose now that there is no $1$ in the quiddity cycle and that $m$ is odd.
Moreover, we may assume that there is a $k$ such that $c_k\ne 0$ or $c_{k+1}\ne 0$ 
and such that $c_kc_{k+1}-1\ne 0$ (otherwise we are finished by Lemma \ref{lem:00111} and Proposition \ref{prop:111}). We consider the quiddity cycle (see Lemma \ref{auvb}, Equation (\ref{eqauvb}))
$$\underline{c}'' = (c_1,\ldots,c_{k-2},c_{k-1}+\frac{1-c_{k+1}}{c_kc_{k+1}-1}, c_kc_{k+1}-1, c_{k+2}+\frac{1-c_k}{c_kc_{k+1}-1},c_{k+3},\ldots,c_m). $$
Using Proposition \ref{cor:t} (notice that $\underline{c}''$ has even length) we get a quiddity cycle 
with a $1$ instead of the entry $c_kc_{k+1}-1$. It has the same length, with the same clusters without zero; we are thus in Case 1 and find a suitable cluster for $\underline{c}''$. In particular, this cluster contains the diagonal corresponding to the entry $c_kc_{k+1}-1$. Extending this triangulation by the diagonal $(k,k+2)$ if $c_k\ne 0$ or $(k+1,k+3)$ if $c_{k+1}\ne 0$ yields a cluster for the original quiddity cycle of length $m$, since the diagonal corresponding to the entry $c_kc_{k+1}-1$ is now $(k,k+3)$.
\end{proof}

\begin{examp}
The following frieze pattern is tame 
but every cluster contains 
a zero entry:
\[
\begin{array}{rrrrrrrrrrrr}
0 & 1 & 0 & -1 & 0 & 1 & 0 &    &    &    &    &   \\
   & 0 & 1 & 0 & -1 & 0 & 1 & 0 &    &    &    &   \\
   &    & 0 & 1 & 0 & -1 & 0 & 1 & 0 &    &    &   \\
   &    &    & 0 & 1 & 0 & -1 & 0 & 1 & 0 &    &   \\
   &    &    &    & 0 & 1 & 0 & -1 & 0 & 1 & 0 &   \\
   &    &    &    &    & 0 & 1 & 0 & -1 & 0 & 1 & 0
\end{array}
\]
By Theorem \ref{thm:clusters} this is (up to repeating this picture periodically) the only such example.
\end{examp}

%%%%%%%%%%%%%%%%%%%%%%%%%%%%%%%%%%%%%%%%%%%%%%%%%%%%%%%%%%%%%%%%%%%%%
%%%%%%%%%%%%%%%%%%%%%%%%%%%%%%%%%%%%%%%%%%%%%%%%%%%%%%%%%%%%%%%%%%%%%
\section{Frieze patterns over $\mathbb{Z}$}
\label{sec:Zfriezes}

In this section we describe the quiddity cycles (and hence tame frieze patterns, cf.\ Proposition \ref{perck})
over the integers $\mathbb{Z}$. As a special case this includes the classic Conway-Coxeter
frieze patterns over $\mathbb{N}$ and also more recent work of Fontaine \cite{F14} on frieze patterns
over $\mathbb{Z}\setminus\{0\}$. It turned out that over $\mathbb{Z}\setminus\{0\}$
only very few new frieze patterns appear in addition to the Conway-Coxeter ones and the new ones 
are closely linked to the old ones. As we will show in this section
the situation changes drastically when zeroes are allowed, i.e.\ when considering quiddity cycles and
frieze patterns over $\mathbb{Z}$. Then a plethora of new frieze patterns emerges. Still, we will 
provide in the next section a nice combinatorial model for obtaining all quiddity cycles over $\mathbb{Z}$
from certain labelled triangulations.
\smallskip

Before we prove the (more technical) theorem (Thm.\ \ref{Zfriezered}) giving reductions of quiddity cycles, we consider a bigger class of cycles which contains our quiddity cycles.

\begin{defin} \label{def:idminusid}
Let $\varepsilon\in \{\pm 1\}$.
An \emph{$\varepsilon$-cycle} is a sequence $(c_1,\ldots,c_m)\in \ZZ^m$ satisfying
\begin{equation}
\prod_{k=1}^{m} \eta(c_k) = \begin{pmatrix}\varepsilon & 0 \\ 0 & \varepsilon \end{pmatrix}.
\end{equation}
\end{defin}

Such cycles with natural numbers as entries have recently been classified by Ovsienko  in terms
of new combinatorial objects called $3d$-dissections \cite{Ovs17}.

\begin{theor}
Let $(c_1,\ldots,c_m)\in \ZZ^m$ be an $\varepsilon$-cycle, $\varepsilon\in\{\pm 1\}$. Then we have at least one of the following cases:
\begin{enumerate}\setcounter{enumi}{-1}
\item\label{I0} $m=2$ and $(c_1,\ldots,c_m)=(0,0)$.
\item\label{I1} There exists an index $k$ with $c_k=1$ and 
$$(c_1,\ldots,c_{k-2},c_{k-1}-1,c_{k+1}-1,c_{k+2}\ldots,c_m)$$
is an $\varepsilon$-cycle (of length $m-1$).
\item\label{I2} There exists an index $k$ with $c_k=0$ and 
$$(c_1,\ldots,c_{k-2},c_{k-1}+c_{k+1},c_{k+2}\ldots,c_m)$$ 
is a $-\varepsilon$-cycle (of length $m-2$).
\item\label{I3} There exists an index $k$ with $c_k=-1$ and 
$$(c_1,\ldots,c_{k-2},c_{k-1}+1,c_{k+1}+1,c_{k+2}\ldots,c_m)$$ 
is a $-\varepsilon$-cycle (of length $m-1$).
\end{enumerate}
\end{theor}
\begin{proof}
By Corollary \ref{cor:small}, there exists a $k$ with $|c_k|<2$. But $c_k\in\ZZ$, thus $c_k\in\{-1,0,1\}$.
If $m=2$ then the cycle is $(0,0)$ (see for example the proof of Cor.\ \ref{cor:small}).
Otherwise, the rules in Equation (\ref{eqa0b})
and Corollary \ref{eta_rule} always apply; when $c_k\in\{-1,0\}$ then a sign appears and an $\varepsilon$-cycle becomes a $-\varepsilon$-cycle.
\end{proof}

To show that there are even reductions producing cycles of the same type ($\varepsilon=-1$, i.e.\ within
the class of quiddity cycles),
we first need the following refinement of Corollary\ \ref{cor:small} in the special case $R=\ZZ$:

\begin{corol} \label{cor:small_Z}
Let $(c_1,\ldots,c_m)\in \ZZ^m$ be a quiddity cycle with $m>3$.
Then there are two indices $j,k\in\{1,\ldots,m\}$ with $|j-k|>1$ and $\{j,k\}\ne\{1,m\}$ such that 
$|c_j|<2$ and $|c_k|<2$.
\end{corol}

\begin{proof}
Before entering the general argument we deal with the case $m=4$ separately. By Example \ref{ex:quiddity}
the quiddity cycles of length 4 have the form $(c_1,\frac{2}{c_1},c_1,\frac{2}{c_1})$ with 
$c_1\neq 0$. Over the integers $\mathbb{Z}$, the only possibilities are $c_1\in \{\pm 1,\pm 2\}$.
It is easy to see that in each of these possible cases the assertion holds.
\smallskip

So from now we assume that $m\ge 5$.
Corollary \ref{cor:small} already gives us two indices $i\ne j$ with $|c_i|<2$ and $|c_j|<2$.
If $|i-j|>1$ and 
$\{i,j\}\ne\{1,m\}$ then we are done.
Otherwise, by rotating the cycle (cf.\ Remark \ref{rem:qcreverse}), 
we may assume without loss of generality
that $i=1$ and $j=m$. We distinguish two cases.

First, if $c_1=c_m=0$, then by the definition of a quiddity cycle we get
\begin{equation}
\label{eq:L3.1}
\prod_{\ell =2}^{m-1} \eta(c_{\ell}) = \eta(0)^{-1} \begin{pmatrix} -1 & 0 \\ 0 & -1 \end{pmatrix}
\eta(0)^{-1} = \begin{pmatrix} 1 & 0 \\ 0 & 1 \end{pmatrix}.
\end{equation}
If $c_{m-1}=0$, we are done. Otherwise, we can apply Lemma \ref{bound2} to the product in equation
(\ref{eq:L3.1}); this 
yields an index $k\in \{3,\ldots,m-2\}$ with $|c_k|<2$ (note that we need our assumption
$m\ge 5$ for this $k$ to exist). Then clearly either $(1,k)$ or $(k,m)$ is a pair of indices as required.

Secondly, otherwise one of $|c_1|$ or $|c_m|$ is $1$.
Reversing the quiddity cycle if required (cf.\ Remark \ref{rem:qcreverse}) 
we may assume that $|c_m|=1$, hence Lemma \ref{bound2} yields a $k\in \{2,\ldots,m-1\}$ with $|c_k|<2$ and again we are finished because $m>3$.
\end{proof}

We can now state the main result of this section, guaranteeing adequate occurrences of entries $0$, $1$ or $-1$ in
quiddity cycles over $\mathbb{Z}$ and providing reductions to shorter quiddity cycles.
This theorem will be the crucial input for the combinatorial model to be presented in the next
section.

\begin{theor} \label{Zfriezered}
Let $(c_1,\ldots,c_m)\in \ZZ^m$ be a quiddity cycle. Then we have at least one of the following cases:
\begin{enumerate}\setcounter{enumi}{-1}
\item\label{T0} $m<4$ and $(c_1,\ldots,c_m)\in \{(0,0),(1,1,1)\}$.
\item\label{T1} There exists an index $k$ with $c_k=1$ and 
$$(c_1,\ldots,c_{k-2},c_{k-1}-1,c_{k+1}-1,c_{k+2}\ldots,c_m)$$
is a quiddity cycle (of length $m-1$).
\item\label{T2} There exists an index $k$ with $c_k=0$, $m$ is odd, and 
$$(-c_1,\ldots,-c_{k-2},-c_{k-1}-c_{k+1},-c_{k+2}\ldots,-c_m)$$ 
is a quiddity cycle (of length $m-2$).
\item\label{T3} There exists an index $k$ with $c_k=-1$, $m$ is even, and 
$$(-c_1,\ldots,-c_{k-2},-c_{k-1}-1,-c_{k+1}-1,-c_{k+2}\ldots,-c_m)$$ 
is a quiddity cycle (of length $m-1$).
\item\label{T4} There exist $j,k$ with $|j-k|>1$, $c_j=c_k=0$ and $$(c_1,\ldots,c_{j-2},c_{j-1}+c_{j+1},c_{j+2}\ldots,c_{k-2},c_{k-1}+c_{k+1},c_{k+2},\ldots,c_m)$$ 
is a quiddity cycle if $|j-k|>2$, or 
$$(c_1,\ldots,c_{j-2},c_{j-1}+c_{j+1}+c_{k+1},c_{k+2},\ldots,c_m)$$ 
is a quiddity cycle if (w.l.o.g.) $j+1=k-1$ (in both cases of length $m-4$).
\item\label{T5} There exist $j,k$ with $|j-k|>1$, $c_j=c_k=-1$
and 
$$(c_1,\ldots,c_{j-2},c_{j-1}+1,c_{j+1}+1,c_{j+2},\ldots,c_{k-2},c_{k-1}+1,c_{k+1}+1,c_{k+2},\ldots,c_m)$$ 
is a quiddity cycle if $|j-k|>2$ or 
$$(c_1,\ldots,c_{j-2},c_{j-1}+1,c_{j+1}+2,c_{k+1}+1,c_{k+2},\ldots,c_m)$$ 
is a quiddity cycle if (w.l.o.g.) $j+1=k-1$ (in both cases of length $m-2$).
\end{enumerate}
\end{theor}
\begin{proof}
If $m<4$, then the only quiddity cycles are $(0,0)$ or $(1,1,1)$ (cf.\ Example \ref{ex:quiddity})
and we are in case (\ref{T0}).
Otherwise, $m>3$ and by
Corollary \ref{cor:small_Z},
there are two entries equal to $-1$, $0$, or $1$, say $c_j$ and $c_k$, with $|j-k|>1$ and $\{j,k\}\ne\{1,m\}$.
If one of them is $1$, then we are in case (\ref{T1}). If $\{c_j,c_k\}=\{0,-1\}$, then depending on whether $m$ is odd or even, we are in case (\ref{T2}) or case (\ref{T3}), respectively. Otherwise, $c_j\ne 1\ne c_k$, $c_j=c_k$, and one of (\ref{T4}) or (\ref{T5}) is the case.

It remains to prove that the reduced sequences given in the respective cases are indeed quiddity cycles.

For case (1) this follows immediately from Corollary \ref{eta_rule}.

For case (2), we have a quiddity cycle $(c_1,\ldots,c_{k-1},0,c_{k+1},\ldots,c_m)$.
Using Remark \ref{rem:qcreverse}\,(2), the assumption $m$ odd and equation (\ref{eqa0b}) we deduce that
\begin{eqnarray*}
\begin{pmatrix} -1 & 0 \\ 0 & -1 \end{pmatrix}
& = & - \eta(-c_1)\ldots\eta(-c_{k-1})\eta(0)\eta(-c_{k+1})\ldots \eta(-c_m) \\
& = & \eta(-c_1)\ldots \eta(c_{k-2})\eta(-c_{k-1}-c_{k+1}) \eta(-c_{k+2})\ldots \eta(-c_m)
\end{eqnarray*}
thus $(-c_1,\ldots,-c_{k-2},-c_{k-1}-c_{k+1},-c_{k+2}\ldots,-c_m)$ is a quiddity cycle.

For case (3), the quiddity cycle has the form $(c_1,\ldots,c_{k-1},-1,c_{k+1},\ldots,c_m)$.
Similar to the previous case, we use Remark \ref{rem:qcreverse}\,(2), the assumption $m$ even 
and Corollary \ref{eta_rule} to obtain
 \begin{eqnarray*}
\begin{pmatrix} -1 & 0 \\ 0 & -1 \end{pmatrix}
& = &  \eta(-c_1)\ldots \eta(-c_{k-2})  \eta(-c_{k-1})\eta(1)\eta(-c_{k+1})\eta(-c_{k+2})
\ldots \eta(-c_m) \\
& = & \eta(-c_1)\ldots \eta(-c_{k-2})\eta(-c_{k-1}-1)\eta(-c_{k+1}-1) \eta(-c_{k+2})\ldots \eta(-c_m)
\end{eqnarray*}
thus $(-c_1,\ldots,-c_{k-2},-c_{k-1}-1,-c_{k+1}-1,-c_{k+2}\ldots,-c_m)$ is a quiddity cycle.

For case (4), we can apply equation (\ref{eqa0b}) twice and the claims follow directly in either case (note that the 
two minus signs from equation (\ref{eqa0b}) cancel).

Finally, for case (5) we apply Corollary \ref{eta_rule}\,(b) twice and the claims about the quiddity cycles follow.
\end{proof}

\begin{examp}
The following examples illustrate some of the cases of Theorem \ref{Zfriezered}.
\begin{enumerate}
\item $m$ even, no $1$ and no $-1$ in the quiddity cycle:
\[ \begin{array}{rrrrrrrrrrrr}
0 & 1 & 0 & -1 & 2 & 1 & 0 &    &    &    &    &   \\
   & 0 & 1 & 2 & -5 & -2 & 1 & 0 &    &    &    &   \\
   &    & 0 & 1 & -2 & -1 & 0 & 1 & 0 &    &    &   \\
   &    &    & 0 & 1 & 0 & -1 & 2 & 1 & 0 &    &   \\
   &    &    &    & 0 & 1 & 2 & -5 & -2 & 1 & 0 &   \\
   &    &    &    &    & 0 & 1 & -2 & -1 & 0 & 1 & 0
\end{array} \]
\item $m$ odd, no $1$ in the quiddity cycle:
\[
\begin{array}{rrrrrrrrrr}
0 & 1 & -1 & 0 & 1 & 0 &    &    &    &   \\
   & 0 & 1 & -1 & -1 & 1 & 0 &    &    &   \\
   &    & 0 & 1 & 0 & -1 & 1 & 0 &    &   \\
   &    &    & 0 & 1 & 0 & -1 & 1 & 0 &   \\
   &    &    &    & 0 & 1 & -1 & 0 & 1 & 0
\end{array}
\]
\item $m$ odd, no $1$ and no $-1$ in the quiddity cycle:
\[ \begin{array}{rrrrrrrrrrrrrrrrrr}
0 & 1 & 0 & -1 & 5 & 1 & -1 & -3 & 1 & 0 &    &    &    &    &    &    &    &   \\
   & 0 & 1 & -4 & 19 & 4 & -3 & -10 & 3 & 1 & 0 &    &    &    &    &    &    &   \\
   &    & 0 & 1 & -5 & -1 & 1 & 3 & -1 & 0 & 1 & 0 &    &    &    &    &    &   \\
   &    &    & 0 & 1 & 0 & -1 & -2 & 1 & -1 & -4 & 1 & 0 &    &    &    &    &   \\
   &    &    &    & 0 & 1 & 4 & 7 & -4 & 5 & 19 & -5 & 1 & 0 &    &    &    &   \\
   &    &    &    &    & 0 & 1 & 2 & -1 & 1 & 4 & -1 & 0 & 1 & 0 &    &    &   \\
   &    &    &    &    &    & 0 & 1 & 0 & -1 & -3 & 1 & -1 & 4 & 1 & 0 &    &   \\
   &    &    &    &    &    &    & 0 & 1 & -3 & -10 & 3 & -2 & 7 & 2 & 1 & 0 &   \\
   &    &    &    &    &    &    &    & 0 & 1 & 3 & -1 & 1 & -4 & -1 & 0 & 1 & 0
\end{array} \]
\item $m$ even, no $1$ and no $0$ in the quiddity cycle:
\[ \begin{array}{rrrrrrrrrrrrrrrr}
0 & 1 & -1 & -3 & 10 & -7 & -3 & 1 & 0 &    &    &    &    &    &    &   \\
   & 0 & 1 & 2 & -7 & 5 & 2 & -1 & 1 & 0 &    &    &    &    &    &   \\
   &    & 0 & 1 & -3 & 2 & 1 & 0 & -1 & 1 & 0 &    &    &    &    &   \\
   &    &    & 0 & 1 & -1 & 0 & 1 & -3 & 2 & 1 & 0 &    &    &    &   \\
   &    &    &    & 0 & 1 & -1 & -3 & 10 & -7 & -3 & 1 & 0 &    &    &   \\
   &    &    &    &    & 0 & 1 & 2 & -7 & 5 & 2 & -1 & 1 & 0 &    &   \\
   &    &    &    &    &    & 0 & 1 & -3 & 2 & 1 & 0 & -1 & 1 & 0 &   \\
   &    &    &    &    &    &    & 0 & 1 & -1 & 0 & 1 & -3 & 2 & 1 & 0
\end{array} \]
\end{enumerate}
\end{examp}

%%%%%%%%%%%%%%%%%%%%%%%%%%%%%%%%%%%%%%%%%%%%%%%%%%%%%%%%%%%%%%%%%%%%%%%%%
%%%%%%%%%%%%%%%%%%%%%%%%%%%%%%%%%%%%%%%%%%%%%%%%%%%%%%%%%%%%%%%%%%%%%%%%%
\section{Combinatorial model}
\label{sec:combmodel}

\begin{defin} \label{def:labelling}
For $m\in \mathbb{N}_{\ge 2}$, let $T$ be a triangulation of a regular $m$-gon. A {\em labelling} of $T$ is an assignment 
of integers $a_t$, called labels, to the triangles $t$ of $T$.
Let $d$ be the sum of the number of negative labels and half the number of labels $0$. If
$d$ is an integer, we call $(-1)^d$ the \emph{sign} of the labelling.
A labelling is called {\em admissible} if 
the following conditions are satisfied:
\begin{enumerate}
\item[{(i)}] The set of triangles $t$ with $a_t\not\in \{1,-1\}$ can be written as a disjoint union of two-element
subsets $\{t_1,t_2\}$ (called \emph{squares}) such that $t_1$ and $t_2$ have a common edge (i.e.\ are neighbouring triangles)
and $a_{t_1}=-a_{t_2}$.
\item[{(ii)}] The sign is $1$, i.e.\ the sum of the number of negative labels and half the number of labels $0$ is even.
\end{enumerate}
\end{defin}
Note that by condition (i), the sign is defined for any admissible labelling.

\begin{figure}
\begin{center}
\begin{minipage}[b]{0.5\textwidth}
\vspace{20pt}
\includegraphics[width=\textwidth]{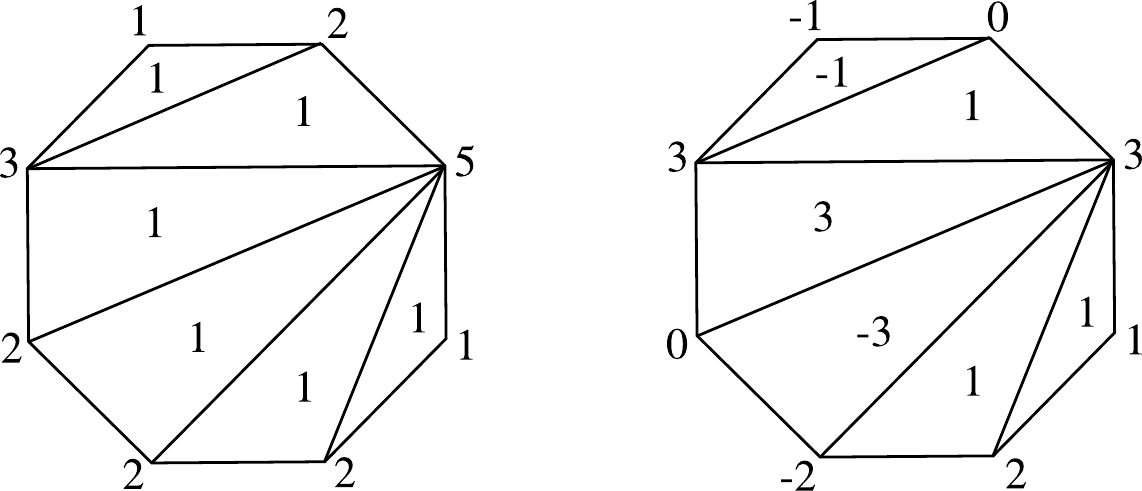}
\end{minipage}
\end{center}
\caption{Two examples for admissible labellings with their quiddity cycles as defined in Theorem \ref{theor:labelquid}\label{excZ}}
\end{figure}
\begin{examp}
Figure \ref{excZ} shows two examples of admissible labellings of a triangulation. The corresponding friezes defined by the quiddity cycles as given in Theorem \ref{theor:labelquid} are:
\[
\begin{array}{rrrrrrrrrrrrrrrr}
0 & 1 & 1 & 2 & 3 & 4 & 5 & 1 & 0 &    &    &    &    &    &    &   \\
   & 0 & 1 & 3 & 5 & 7 & 9 & 2 & 1 & 0 &    &    &    &    &    &   \\
   &    & 0 & 1 & 2 & 3 & 4 & 1 & 1 & 1 & 0 &    &    &    &    &   \\
   &    &    & 0 & 1 & 2 & 3 & 1 & 2 & 3 & 1 & 0 &    &    &    &   \\
   &    &    &    & 0 & 1 & 2 & 1 & 3 & 5 & 2 & 1 & 0 &    &    &   \\
   &    &    &    &    & 0 & 1 & 1 & 4 & 7 & 3 & 2 & 1 & 0 &    &   \\
   &    &    &    &    &    & 0 & 1 & 5 & 9 & 4 & 3 & 2 & 1 & 0 &   \\
   &    &    &    &    &    &    & 0 & 1 & 2 & 1 & 1 & 1 & 1 & 1 & 0
\end{array}
\]
\[
\begin{array}{rrrrrrrrrrrrrrrr}
0 & 1 & -1 & -4 & 1 & 2 & 3 & 1 & 0 &    &    &    &    &    &    &   \\
   & 0 & 1 & 3 & -1 & -1 & -1 & 0 & 1 & 0 &    &    &    &    &    &   \\
   &    & 0 & 1 & 0 & -1 & -2 & -1 & -1 & 1 & 0 &    &    &    &    &   \\
   &    &    & 0 & 1 & -2 & -5 & -3 & -4 & 3 & 1 & 0 &    &    &    &   \\
   &    &    &    & 0 & 1 & 2 & 1 & 1 & -1 & 0 & 1 & 0 &    &    &   \\
   &    &    &    &    & 0 & 1 & 1 & 2 & -1 & -1 & -2 & 1 & 0 &    &   \\
   &    &    &    &    &    & 0 & 1 & 3 & -1 & -2 & -5 & 2 & 1 & 0 &   \\
   &    &    &    &    &    &    & 0 & 1 & 0 & -1 & -3 & 1 & 1 & 1 & 0
\end{array}
\]
\end{examp}

\begin{theor} \label{theor:labelquid}
\begin{enumerate}
\item[{(a)}] 
Let $T$ be a triangulation of a regular $m$-gon with vertices denoted (in counterclockwise order) 
$1,2,\ldots, m$, and assume that we have an
admissible labelling of $T$. For each vertex $i$ let $c_i$ be the sum of the labels 
of the triangles attached at the vertex $i$. Then $(c_1,\ldots,c_{m})$ is a quiddity cycle over $\mathbb{Z}$
(in the sense of Definition \ref{def:quiddity}).
\item[{(b)}] Every quiddity cycle over $\mathbb{Z}$ can be obtained as in (a) from an admissible labelling.
\end{enumerate}
\end{theor}

\begin{proof}
(b) By Theorem \ref{Zfriezered}, there is a (not necessarily unique) sequence of transformations which reduces a quiddity cycle over $\mathbb{Z}$ to $(0,0)$ or $(1,1,1)$. Actually, $(0,0)$ suffices since $(1,1,1)$ can be reduced to $(0,0)$ by Corollary \ref{eta_rule} (a).
We translate each of these transformations into combinatorial pictures and obtain the figures \ref{FT1}, \ref{FT2}, \ref{FT3}, \ref{FT4}, and \ref{FT5}. (We write ``$\cdot (-1)$'' into the polygon to indicate that all entries are multiplied by $-1$.) To obtain an admissible labelling for a given quiddity cycle, just apply such a sequence of transformations in the reverse ordering to the $2$-gon.\\
(a) Let $(a_t)_{t\in T}$ be an admissible labelling of $T$.
This may be constructed inductively (although not uniquely) by successively gluing building blocks, i.e.\ triangles with entries $1$ or $-1$ and squares containing $c,-c$ for some values $c$. We start with the ``triangulation'' of the $2$-gon and labels $0,0$ at the vertices; this is admissible and gives the quiddity cycle $(0,0)$.
In each step, we glue a building block at an edge $(i,i+1)$ of a triangulation with vertices labelled as above (sum of the labels of the triangles), say $a=c_i$, $b=c_{i+1}$. For each type of block, this produces a new sequence of labels at the vertices. However, if the sign of the obtained labelling is $\varepsilon$, then the product of matrices of the form $\eta(\cdot)$ of the labels at the vertices is $\varepsilon$ times the identity matrix, thus the labels are a quiddity cycle only if $\varepsilon=1$:
\begin{itemize}[leftmargin=15pt]
\item
If we glue a triangle with a $1$: $(\ldots,a,b,\ldots)$ becomes $(\ldots,a+1,1,b+1,\ldots)$ which is compatible with $\eta(a)\eta(b)=\eta(a+1)\eta(1)\eta(b+1)$ (Corollary \ref{eta_rule}), the sign remains the same.
\item
If we glue a triangle with a $-1$: $(\ldots,a,b,\ldots)$ becomes $(\ldots,a-1,-1,b-1,\ldots)$.
Here the sign changes: Corollary \ref{eta_rule} tells us $\eta(a)\eta(b)=-\eta(a-1)\eta(-1)\eta(b-1)$.
\item
If we glue a square containing $c,-c$: $(\ldots,a,b,\ldots)$ becomes $(\ldots,a,c,0,b-c,\ldots)$ or $(\ldots,a-c,0,c,b,\ldots)$.
Here the sign changes as well: Equation (\ref{eqa0b}) tells us
$\eta(a)\eta(c)\eta(0)\eta(b-c) = -\eta(a)\eta(b) =  \eta(a-c)\eta(0)\eta(c)\eta(b)$.
\end{itemize}
In any case, the sign of the labelling is $1$ if and only if the obtained sequence is a quiddity cycle.
\end{proof}

\begin{figure}
\begin{center}
\begin{minipage}[b]{0.7\textwidth}
\vspace{20pt}
\includegraphics[width=\textwidth]{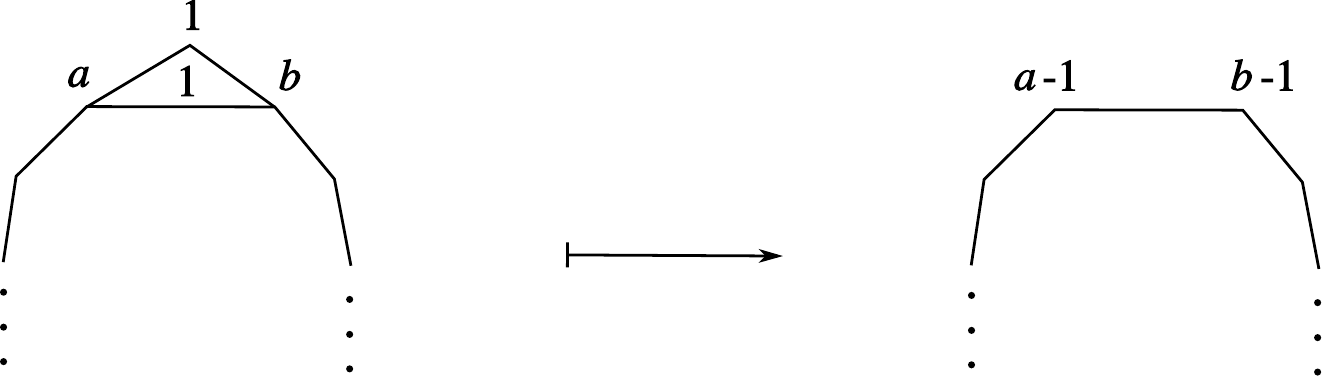}
\end{minipage}
\end{center}
\caption{Theorem \ref{Zfriezered}, case (\ref{T1}) \label{FT1}}
\end{figure}
\begin{figure}
\begin{center}
\begin{minipage}[b]{0.7\textwidth}
\includegraphics[width=\textwidth]{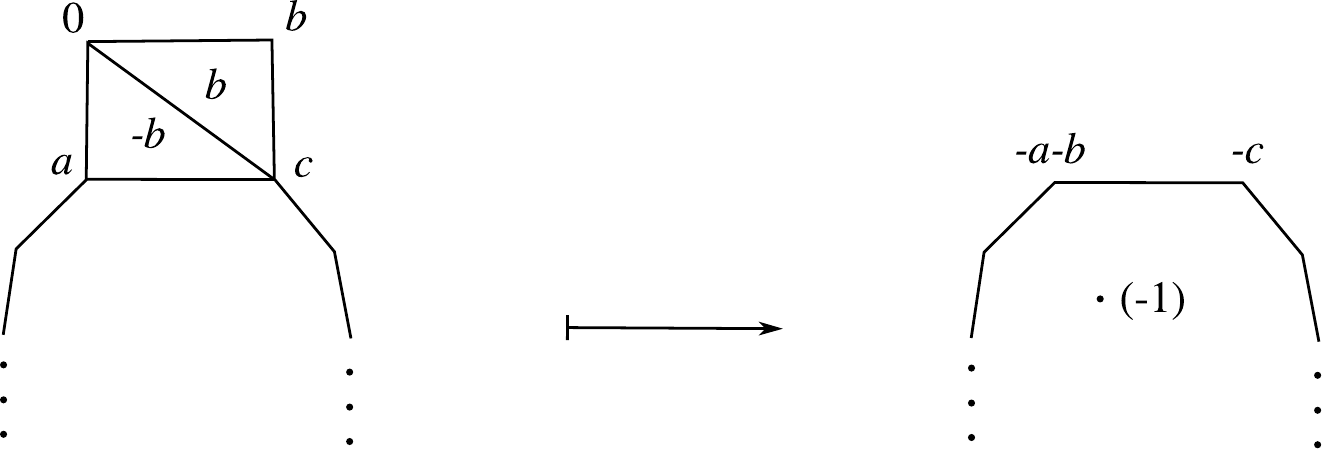}
\end{minipage}
\end{center}
\caption{Theorem \ref{Zfriezered}, case (\ref{T2}) ($m$ odd) \label{FT2}}
\end{figure}
\begin{figure}
\begin{center}
\begin{minipage}[b]{0.7\textwidth}
\includegraphics[width=\textwidth]{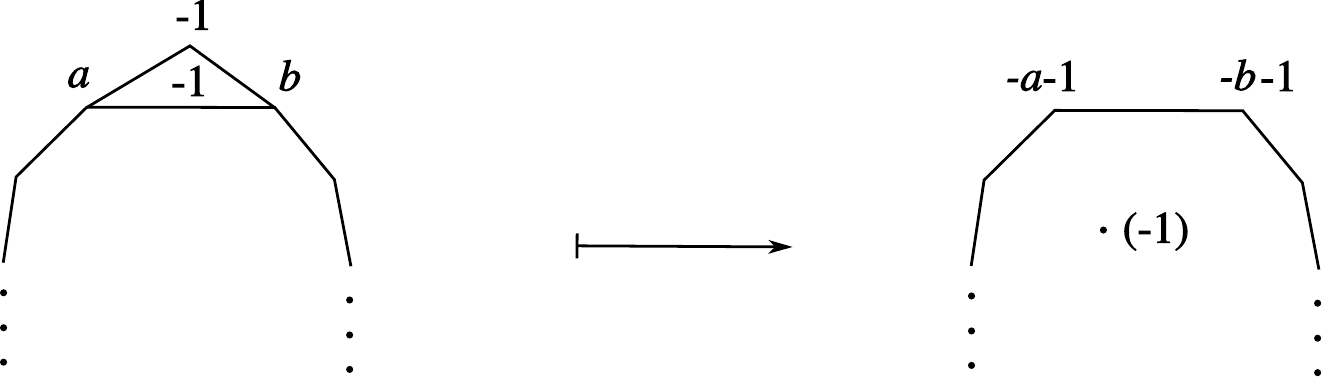}
\end{minipage}
\end{center}
\caption{Theorem \ref{Zfriezered}, case (\ref{T3}) ($m$ even) \label{FT3}}
\end{figure}
\begin{figure}
\begin{center}
\begin{minipage}[b]{0.7\textwidth}
\includegraphics[width=\textwidth]{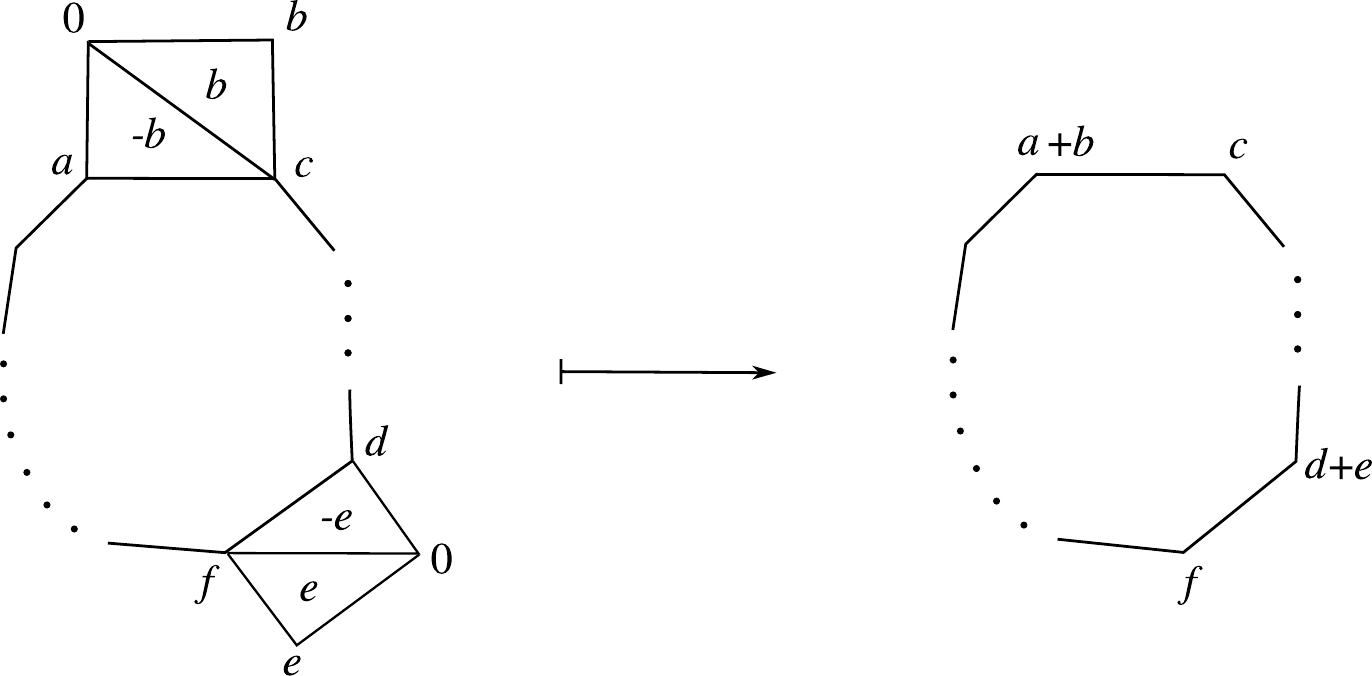}
\end{minipage}
\end{center}
\caption{Theorem \ref{Zfriezered}, case (\ref{T4}) \label{FT4}}
\end{figure}
\begin{figure}
\begin{center}
\begin{minipage}[b]{0.7\textwidth}
\includegraphics[width=\textwidth]{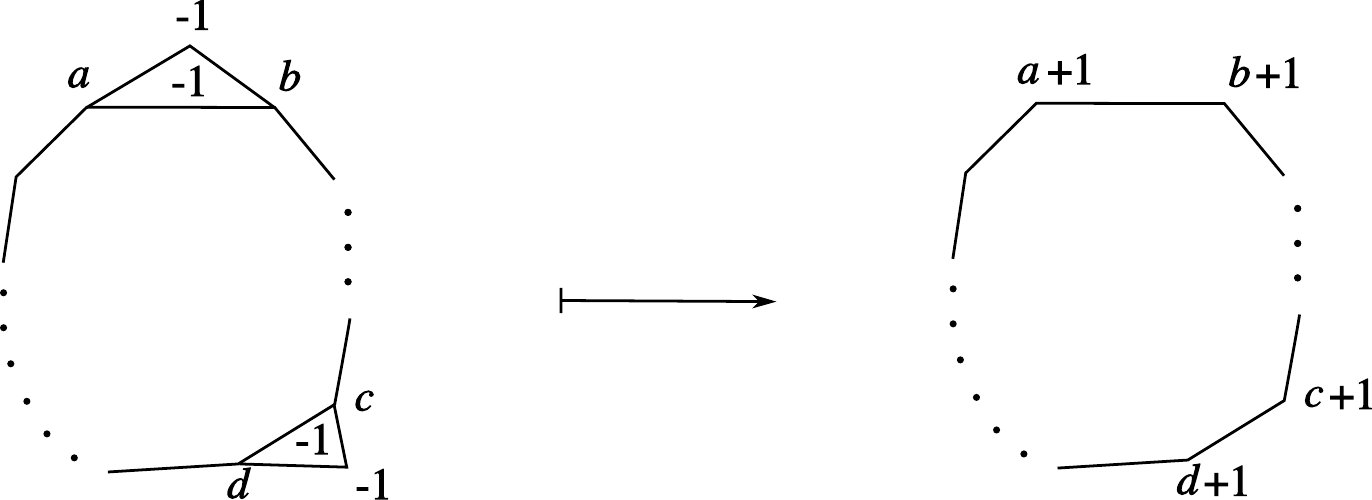}
\end{minipage}
\end{center}
\caption{Theorem \ref{Zfriezered}, case (\ref{T5}) ($m$ odd) \label{FT5}}
\end{figure}

\begin{remar}
(1) The construction described above directly generalizes the classic case of Conway-Coxeter frieze patterns 
\cite{CC73}.
In fact, in the classic case, the quiddity cycle of a frieze pattern over $\mathbb{N}$ is determined from the 
corresponding triangulation by counting the number of triangles attached to each vertex. In other words,
in the classic case one only considers admissible labellings where each triangle is labelled by 1.
\smallskip

\noindent
(2) Fontaine \cite{F14} has described frieze patterns over $\mathbb{Z}\setminus \{0\}$. His main theorem 
in this direction states that for even height $n$ each such frieze pattern is a Conway-Coxeter frieze pattern, and for
$n$ odd there are twice as many frieze patterns, obtained from the Conway-Coxeter frieze patterns by 
multiplying each second row by $-1$. In particular, in the latter case, the quiddity cycle gets multiplied
by $-1$. So in the above language of labellings this means that for $n$ odd the labellings for the 
new frieze patterns are constantly $-1$ on each triangle. (Note that for $n$ odd, the corresponding triangulation
of the $(n+3)$-gon has $n+1$ triangles; this is an even number, so condition (ii) of Definition \ref{def:labelling}
is satisfied and this labelling is admissible.)   
\smallskip

\noindent
(3) Theorem \ref{theor:labelquid} shows that in addition to the frieze patterns considered by Conway-Coxeter
and Fontaine there are infintely many other tame frieze patterns over $\mathbb{Z}$ (but these would all have at least
one zero entry). More precisely, there are four tame frieze patterns of height 1 over $\mathbb{Z}$ (all 
without zero entry). For every height $n\ge 2$ there are infinitely many tame frieze patterns over $\mathbb{Z}$.  In fact,
every regular $(n+3)$-gon for $n\ge 2$ can be given an admissible triangulation with two neighbouring triangles labelled 
by $\pm a\in \mathbb{Z}$. Then Theorem \ref{theor:labelquid} yields infinitely many quiddity cycles over $\mathbb{Z}$
of length $n+3$. Then Proposition \ref{perck}) gives infinitely many corresponding tame frieze
patterns over $\mathbb{Z}$ of height $n$. 
\smallskip

\noindent
(4) In \cite{Ovs17}, Ovsienko proves that there is a bijection between quiddity cycles with entries in $\mathbb{N}$ 
and $3d$-dissections, i.e. triangulations of regular polygons into subpolygons whose numbers of vertices are 
multiples of 3. This nicely generalizes the Conway-Coxeter bijection between quiddity cycles of frieze patterns over 
$\mathbb{N}$ and triangulations of polygons. Our Theorem \ref{theor:labelquid} yields all quiddity cycles over
$\mathbb{Z}$, via admissible labellings of triangulations. For the subset of quiddity cycles over $\mathbb{N}$ it is
not hard to show how the different combinatorial models, $3d$-dissections and admissible labellings are related, that is
how to obtain an admissible labelling from any $3d$-dissection. Since a precise description needs some technicalities
we do not address this issue here further.   
\smallskip

\noindent
(5) Theorem \ref{theor:labelquid} provides a combinatorial model for producing all quiddity cycles over 
$\mathbb{Z}$. Unfortunately, in this way 
one does not get a bijection between quiddity cycles and admissible labellings
of triangulations. There exist different admissible labellings giving the same quiddity cycle, as the following 
hexagonal example shows; for any $a\in \mathbb{Z}$, 
both admissible labellings yield the quiddity cycle $(1,1,-a,-1,-1,a)$ (starting from the 
top left vertex).  
\begin{center}
\begin{figure}[H]
  \centering
  \begin{minipage}[b]{0.4\textwidth}
   \begin{tikzpicture}[auto]
    \node[name=s, draw, shape=regular polygon, regular polygon sides=6, minimum size=3.0cm] {};
    \draw[thick] (s.corner 2) to (s.corner 4);
    \draw[thick] (s.corner 6) to (s.corner 2);
    \draw[thick] (s.corner 6) to (s.corner 4);
   \draw[shift=(s.corner 1)]  node[label={[shift={(-0.2,-0.7)}]{\small $a$}}]  {};
  \draw[shift=(s.corner 3)]  node[label={[shift={(0.5,-0.4)}]{\small $1$}}]  {};
  \draw[shift=(s.corner 5)]  node[label={[shift={(-0.12,0.06)}]{\small $-1$}}]  {};
  \draw[shift=(s.corner 6)]  node[label={[shift={(-1.5,-0.4)}]{\small $-a$}}]  {};
   \end{tikzpicture}
   \end{minipage}
   \begin{minipage}[b]{0.4\textwidth}
   \begin{tikzpicture}[auto]
    \node[name=s, draw, shape=regular polygon, regular polygon sides=6, minimum size=3.0cm] {};
    \draw[thick] (s.corner 1) to (s.corner 3);
    \draw[thick] (s.corner 6) to (s.corner 3);
    \draw[thick] (s.corner 6) to (s.corner 4);
   \draw[shift=(s.corner 1)]  node[label={[shift={(-0.2,-1.1)}]{\small $a-1$}}]  {};
  \draw[shift=(s.corner 2)]  node[label={[shift={(0.1,-0.7)}]{\small $1$}}]  {};
  \draw[shift=(s.corner 3)]  node[label={[shift={(1.0,-0.8)}]{\small $-a+1$}}]  {};
  \draw[shift=(s.corner 5)]  node[label={[shift={(-0.12,0.06)}]{\small $-1$}}] {} ;
   \end{tikzpicture}
   \end{minipage}
\end{figure}
\end{center}
It seems to be tricky to describe combinatorially precisely when two admissible labellings give the same
quiddity cycle.   
\end{remar}

We now prove an analogue to Theorem \ref{Zfriezered} for the combinatorial model.
The following lemma is the key to obtain all required reductions.

\begin{lemma}\label{L12131}
Let $\underline{c} = (c_1,...,c_m)$, $m>3$ be a quiddity cycle of a triangulation, i.e.\ the quiddity cycle of the Conway-Coxeter frieze pattern of this triangulation. Then $\underline{c}$ contains two disjoint subsequences $(1,2)$ or $(2,1)$, or it contains the subsequence $(1,3,1)$.
\end{lemma}

\begin{proof}
Assume that $(1,3,1)$ is not a subsequence of $\underline{c}$.
Now consider all subsequences in $\underline{c}$ of the form $(a,1,b)$ with $a,b>2$. If $(a,1,b,1,d)$ with $a,b,d>2$ is a subsequence, then $b>3$ by assumption. Thus we may apply Corollary \ref{eta_rule} (a) to all these subsequences simultaneously and obtain a new quiddity cycle $\underline{c}'$ in which $(a,1,b)$ with $a,b>2$ is never a subsequence;
this $\underline{c}'$ has length at least $4$ because it contains at least one entry bigger than $1$.
In particular, $(1,1)$ is not a subsequence of $\underline{c}'$.

Hence every $1$ in $\underline{c}'$ has a neighbouring $2$, and we have at least two $1$'s since the length of $\underline{c}'$ is greater than $3$.
If the $1$'s are at positions $k$ and $k+2$ and the neighbouring $2$'s are both at $k+1$, i.e.\ $(1,2,1)$ is a subsequence, then $\underline{c}'=(1,2,1,2)$. Otherwise we find two disjoint subsequences in $\underline{c}'$.
Now we go back to $\underline{c}$: including the $1$'s we have removed, the disjoint subsequences found in $\underline{c}'$ remain, since $(1,3,1)$ is not a subsequence of $\underline{c}$.
\end{proof}

Given a triangulation of a polygon, a vertex of the polygon is called an \emph{ear}, if it is attached to only one triangle.

The following theorem is a combinatorial reformulation of Theorem \ref{Zfriezered}.

\begin{theor}
Let $T$ be a triangulation of a regular $m$-gon with an admissible labelling.
Then we have at least one of the following cases:
\begin{enumerate}\setcounter{enumi}{-1}
\item\label{TC0} $m<4$ and the labelling is either empty or it consists of one entry equal to $1$.
\item\label{TC1} There exists an index $k$ such that $(k-1,k,k+1)$ is a triangle labelled $1$.
Removing this triangle yields an admissible labelling of a triangulation of an $(m-1)$-gon.
\item\label{TC2} There exists an index $k$ such that $(k-1,k,k+1,k+2)$ is a square labelled $(c,-c)$ and $m$ is odd.
Removing this square and multiplying each entry by $-1$ yields an admissible labelling of a triangulation of an $(m-2)$-gon.
\item\label{TC3} There exists an index $k$ such that $(k-1,k,k+1)$ is a triangle labelled $-1$ and $m$ is even.
Removing this triangle and multiplying each entry by $-1$ yields an admissible labelling of a triangulation of an $(m-1)$-gon.
\item\label{TC4} There exist $j,k$ with $|j-k|>1$ and $(j-1,j,j+1,j+2)$, $(k-1,k,k+1,k+2)$ are squares labelled $(c,-c)$, $(d,-d)$.
Removing these two squares yields an admissible labelling of a triangulation of an $(m-4)$-gon.
\item\label{TC5} There exist $j,k$ with $|j-k|>1$ and $(j-1,j,j+1)$, $(k-1,k,k+1)$ are triangles labelled $-1$.
Removing these two triangles yields an admissible labelling of a triangulation of an $(m-2)$-gon.
\end{enumerate}
\end{theor}

\begin{proof}
If $m<4$ then we are in case (\ref{TC0}). Thus assume that $m>3$.
The triangulation $T$ has an ear, thus if it is labelled $1$, then we have case (\ref{TC1}). Thus assume that no ear is labelled $1$.
By Lemma \ref{L12131}, we have the following cases for the quiddity cycle $q$ of the Conway-Coxeter frieze pattern of the triangulation $T$:

1. The cycle $q$ contains $(1,3,1)$, say at position $k,k+1,k+2$. If the pentagon $(k-1,k,k+1,k+2,k+3)$ has no triangle with label $1$, then either we have two ears labelled $-1$ and we are in case (\ref{TC5}), or the pentagon consists of a triangle labelled $-1$ and a square labelled $(c,-c)$, we are then in case (\ref{TC2}) or (\ref{TC3}) depending on whether $m$ is odd or even:
If $m$ is odd, remove the triangle labelled $-1$; the sign of the smaller labelling is different, but since $m$ is odd, multiplying each entry by $-1$ remedies this. If $m$ is even, remove the square; again the sign changes and is remedied by multiplying each entry by $-1$.

2. The cycle $q$ is equal to $(1,2,1,2)$. Then both triangles are labelled $-1$ because this is the only admissible labelling without a $1$, and we are in case (\ref{TC3}).

3. The cycle $q$ contains two disjoint subsequences $(1,2)$ or $(2,1)$ and $m>4$. If both ears are labelled $-1$ then we are in case (\ref{TC5}). If both subsequences belong to squares labelled $(c,-c)$, $(d,-d)$, then we are in case (\ref{TC4}).
Otherwise, we are in case (\ref{TC2}) or (\ref{TC3}) depending on whether $m$ is odd or even (same argument as in case 1 to see that the new labelling is admissible).
\end{proof}

\end{document}